\documentclass[a4paper,reqno]{amsart}
\usepackage[british]{babel}
\usepackage{amssymb,amsmath,amsbsy,mathtools}
\usepackage{hyphenat}
\usepackage{tikz}
\usepackage{upgreek}
\usepackage{dsfont}
\usepackage{graphicx,color,nicefrac} 
\usepackage{amsfonts}
\usepackage[normalem]{ulem}
\usepackage[export]{adjustbox}
\usepackage[hidelinks]{hyperref}
\usepackage{csquotes}

\newtheorem{theorem}{Theorem}[section]
\newtheorem{lemma}[theorem]{Lemma}
\newtheorem{corollary}[theorem]{Corollary}

\newtheorem{remark}[theorem]{Remark}

\theoremstyle{definition}
\newtheorem{definition}[theorem]{Definition}
\newtheorem{assumption}[theorem]{Assumption}

\newcommand{\Det}{\operatorname{det}}
\newcommand{\dint}{\thickspace \mathrm{d}}
\renewcommand{\vec}{\boldsymbol}
\renewcommand{\subset}{\subseteq}

\renewcommand{\epsilon}{\varepsilon}

\numberwithin{equation}{section}
\numberwithin{theorem}{section}
\usepackage[foot]{amsaddr}

\usetikzlibrary{positioning}

\newcommand{\Nbbb}{\mathbb{N}}
\newcommand{\Rbbb}{\mathbb{R}}

\newcommand{\Abfm}{\mathbf{A}}
\newcommand{\Dbfm}{\mathbf{D}}
\newcommand{\Lbfm}{\mathbf{L}}

\newcommand{\dbfm}{\mathbf{d}}
\newcommand{\gbfm}{\mathbf{g}}
\newcommand{\jbfm}{\mathbf{j}}
\newcommand{\kbfm}{\mathbf{k}}
\newcommand{\qbfm}{\mathbf{q}}
\newcommand{\sbfm}{\mathbf{s}}
\newcommand{\tbfm}{\mathbf{t}}
\newcommand{\ubfm}{\mathbf{u}}
\newcommand{\vbfm}{\mathbf{v}}

\newcommand{\Bcal}{\mathcal{B}}
\newcommand{\Fcal}{\mathcal{F}}
\newcommand{\Ical}{\mathcal{I}}
\newcommand{\Jcal}{\mathcal{J}}
\newcommand{\Lcal}{\mathcal{L}}
\newcommand{\Ocal}{\mathcal{O}}

\newcommand{\Bfrak}{\mathfrak{B}}
\newcommand{\Hfrak}{\mathfrak{H}}

\newcommand{\xvec}{\vec{x}}

\newcommand{\jx}{j_\mathrm{x}}
\newcommand{\jy}{j_\mathrm{y}}
\newcommand{\kx}{k_\mathrm{x}}
\newcommand{\ky}{k_\mathrm{y}}
\newcommand{\mx}{\mathrm{m}_\mathrm{x}}
\newcommand{\my}{\mathrm{m}_\mathrm{y}}
\newcommand{\sx}{s_\mathrm{x}}
\newcommand{\sy}{s_\mathrm{y}}
\newcommand{\tx}{t_\mathrm{x}}
\newcommand{\ty}{t_\mathrm{y}}

\newcommand{\Mx}{\mathrm{M}_{\mathrm{x}}}
\newcommand{\My}{\mathrm{M}_{\mathrm{y}}}

\newcommand{\deltax}{\delta_\mathrm{x}}
\newcommand{\deltay}{\delta_\mathrm{y}}

\newcommand{\lambdax}{\lambda_\mathrm{x}}
\newcommand{\lambday}{\lambda_\mathrm{y}}
\newcommand{\sigmax}{\sigma_\mathrm{x}}
\newcommand{\sigmay}{\sigma_\mathrm{y}}

\newcommand{\minm}{\mathrm{m}}
\newcommand{\maxm}{\mathrm{M}}

\newcommand{\fxt}{\tilde{f}_{\mathrm{x}}}
\newcommand{\fxb}{\bar{f}_{\mathrm{x}}}

\newcommand{\fyb}{\bar{f}_{\mathrm{y}}}

\newcommand{\kappaxt}{\tilde{\kappa}_{\mathrm{x}}}

\newcommand{\alphabfm}{\boldsymbol{\upalpha}}
\newcommand{\gammabfm}{\boldsymbol{\upgamma}}
\newcommand{\lambdabfm}{\boldsymbol{\lambda}}
\newcommand{\zerobfm}{\boldsymbol{0}}
\newcommand{\onebfm}{\boldsymbol{1}}
\newcommand{\nablabfm}{\boldsymbol{\nabla}}

\newcommand{\tauvec}{\vec{\tau}}
\newcommand{\xivec}{\vec{\xi}}

\newcommand{\Dparam}[2]{\mathcal{D}_{#1,#1'}^{({\rm #2})}}
\newcommand{\Eparam}[2]{\mathcal{E}_{#1,#1'}^{({\rm #2})}}
\newcommand{\Fparam}[2]{\mathcal{F}_{#1,#1'}^{({\rm #2})}}

\DeclareMathOperator{\Diam}{diam}
\DeclareMathOperator{\Dist}{dist}
\DeclareMathOperator*{\esssup}{ess\,sup}
\DeclareMathOperator{\Sign}{sign}
\DeclareMathOperator{\Singsupp}{sing\,supp}
\DeclareMathOperator{\Supp}{supp}

\newcommand{\isdef}{\mathrel{\mathrel{\mathop:}=}}

\DeclarePairedDelimiter\groupa{\langle}{\rangle}%
\DeclarePairedDelimiter\groupp{(}{)}%
\DeclarePairedDelimiter\norms{\lvert}{\rvert}%

\definecolor{cmred}{RGB}{190,64,64}
\definecolor{cmblue}{RGB}{64,97,190}

\title[Integral Operators in Anisotropic Wavelet
Coordinates]{On the Compressibility of Integral Operators in Anisotropic Wavelet
Coordinates}
\author{Helmut Harbrecht}
\author{Remo von Rickenbach}
\address{ Universit\"at Basel\\
	Departement Mathematik und Informatik\\
	Spiegelgasse 1\\
	4051 Basel\\
Switzerland}
\email{\{helmut.harbrecht,remo.vonrickenbach\}@unibas.ch}

\begin{document}

\begin{abstract}
	The present article is concerned with the \(s^\star\)-compressibility
	of classical boundary integral operators which are discretised by 
	using anisotropic tensor-product wavelets as trial functions.
	Having the \(s^\star\)-compressibility at hand, 
	one can design adaptive wavelet algorithms 
	which are asymptotically optimal, 
	meaning that any target accuracy can be achieved at a computational expense 
	that stays proportional to the number of degrees of freedom 
	(within the setting determined by an underlying wavelet basis) 
	that would ideally be necessary for realising that target accuracy 
	if full knowledge about the unknown solution were given. 
	As we consider anisotropic wavelet bases,
	we can achieve higher convergence rates compared to the standard, 
	isotropic setting. 
	Especially, edge singularities of anisotropic nature can be resolved.
\end{abstract}

\maketitle

\section{Introduction}\label{sec:Introduction}

Many problems from physics and engineering require the solution of an operator
equation
\begin{equation}
	\Lcal u = f,
	\label{eq:operator_equation}
\end{equation}
where \(\Lcal\colon V \to V'\) is a continuous and uniformly elliptic operator.
For the approximation of the unknown solution \(u \in V\),
there is a variety of approaches and algorithms which are well-understood.
In particular, if the solution \(u\) is sufficiently smooth
and the operator \(\Lcal\) is local,
finite element approximation schemes \cite{Bra13,BS08}
combined with multigrid methods \cite{BH83,Hac81}
deliver approximants within linear computing time.
On the other hand, if the operator \(\Lcal\) is nonlocal,
there is a need for approximating the system matrix resulting from a Galerkin
discretisation. Adaptive cross approximation \cite{BR03}, panel clustering
\cite{HN89,Sau00}, or the fast multipole method \cite{GR87} result in linear 
or log-linear complexity. Moreover, wavelet matrix compression is known 
to admit linear complexity for both isotropic \cite{DHS06,Sch98} and 
anisotropic tensor-product wavelets
\cite{HvR24}.

However, if the solution \(u\) provides only limited regularity, the 
classical approaches based on uniform mesh refinement usually do
not converge with the best possible rate since local mesh refinement 
is not possible in this \emph{linear approximation} method.
Therefore, there is a need for adaptive schemes in this case, 
which fall under the framework of \emph{nonlinear
approximation} \cite{DeV98}, 
where the term nonlinear refers to the nonlinearity
of the trial spaces involved.
An example of nonlinear approximation is \emph{best \(N\)-term approximation},
where a function is approximated from a trial space consisting of linear 
combinations of at most \(N\) terms from a dictionary.
Since in this setting, 
all possible approximations from linear trial spaces of dimension \(N\)
can be constructed,
it is immediately clear that the best nonlinear approximant is at least 
as accurate as the best linear approximant.

The dictionary involved can be realised by a wavelet basis,
cf.\ \cite{Dah97,Dau92,DeV98,Sch98}.
Since wavelet bases are especially Riesz bases, 
approximating a function is equivalent to approximating its coefficient vector.
This means in particular that (a rearrangement by decreasing absolute value of)
the coefficients of a function \(u\) must decay at a certain rate to
ensure that \(u\) is well-approximable.
In general, the required decay to achieve a rate \(N^{-s}\) with \(N\) terms 
is much weaker for a nonlinear approximation scheme than 
for a linear approximation scheme, 
cf.\ \cite{DeV98} and the references therein.

If \eqref{eq:operator_equation} is discretised by a Riesz basis, the resulting
infinite linear system of equations, involving a 
bi-infinite Galerkin matrix \(\Lbfm\), is well-conditioned. 
On the assumption that the bi-infinite matrix \(\Lbfm\) is fully known, 
or equivalently, 
that each matrix entry can be calculated in \(\Ocal(1)\) operations, 
algorithms were developed which approximate the 
coefficient vector \(\ubfm\) with \(N\) terms at the rate \(N^{-s}\) 
whenever \(u\in V\) allows for this rate, cf.\ the original articles
\cite{CDD01,CDD02} and also the overview article \cite{Ste09}.
In particular, this means that we can approximate the solution \(u\) of
\eqref{eq:operator_equation} at the same rate as if full knowledge on
\(u\) was provided.
The requirement that such algorithms can be realised is that the 
bi-infinite matrix \(\Lbfm\) is sufficiently compressible. 
More precisely, \(\Lbfm\) must be
\emph{\(s^{\star}\)-compressible}, meaning that 
for any \(r\in\Nbbb\) there must exist a bi-infinite 
matrix \(\Lbfm_r\) such that in every row and column of \(\Lbfm_r\)
there are asymptotically only \(\alpha_r 2^r\) nontrivial entries, 
and that the overall error \(\|\Lbfm -\Lbfm_r\|_2\) decays at the 
rate \(\beta_r 2^{-sr}\) for any \(s < s^{\star}\). 
Herein, \((\alpha_r)_r\) and \((\beta_r)_r\) 
denote arbitrary summable sequences.

By construction, the best possible approximation rate in the
Sobolev space \(H^{q}\), in \(n\) spatial dimensions, is 
\(N^{-\frac{d-q}{n}}\) for piecewise polynomial trial functions 
of order \(d\), so if \(\Lbfm\) is \emph{\(s^\star\)-compressible} 
for some \(s^\star > \frac{d-q}{n}\), the solution \(u\) of 
\eqref{eq:operator_equation} can be approximated in \(\Ocal(N)\) 
operations at the rate of the best \(N\)-term approximation 
which can be achieved with the given basis at all. This 
is hence \emph{optimal}. Note that realisations of respective adaptive
algorithms can be found in 
\cite{CDD01,CDD02,DSS09,GHS07,KSU16,KS13,KS14,SS09},
for example.

For an isotropic wavelet basis \(\Psi\), \(s^\star\)-compressibility
has been verified in \cite{Ste04}, 
where \(\Lcal\) cannot only be a partial differential operator, 
but also a singular boundary integral operator as it is considered here. 
Since singular boundary integral operators are nonlocal, 
the associated matrix \(\Lbfm\) is densely populated,
so this seems surprising at the first glance.
However, because the trial functions involved are wavelets,
they admit a \emph{cancellation property} which implies that the entries
decay with the difference in levels of the involved wavelets.
Of course, there are certain requirements to the order \(d\) 
and the number of vanishing moments \(\tilde{d}\) of the 
wavelet basis \(\Psi\), cf.\ \cite{Ste04} for the details. Nonetheless,
it was shown in \cite{DHS07,GS06} that optimality can indeed be 
realised numerically. 
Earlier, under higher regularity assumptions to the solution, 
the same was verified also for a linear approximation scheme \cite{DHS06}.

The drawback of an isotropic wavelet basis is that singularities of
anisotropic kind, such as inward-directed edges, 
cannot be resolved optimally, see e.g.\ \cite{HU18,Utz16}.
Therefore, it seems natural to use 
anisotropic wavelets to discretise \(V\). Anisotropic wavelets for 
boundary integral equations were first considered in \cite{GOS99} 
in the context of sparse tensor product spaces. In \cite{BHW23,HvR24b}, 
it was shown that the target function \(u\) must admit hybrid 
Besov regularity to achieve a certain convergence rate, 
weakening the requirements for an approximation order \(N^{-s}\) 
when compared to the isotropic setting, 
wherein isotropic Besov regularity of a higher order is required.

In \cite{SS08}, optimal compressibility and computability has 
been verified for anisotropic wavelets for a partial differential 
operator on a product domain. 
For a singular boundary integral operator, on the other hand, 
this question has not been answered yet,
although this case covers the so-called boundary element method 
\cite{SS11,Ste08} and is therefore highly relevant in practice.

Recently, there was some progress in this direction.
If one discretises \eqref{eq:operator_equation}
using anisotropic tensor-product wavelets,
the resulting Galerkin matrix can be optimally compressed,
meaning that there are only \(\Ocal(N)\) nontrivial entries left
in the compressed matrix.
Meanwhile, the convergence rate with respect to the uncompressed scheme is
left unchanged.
This was shown by the authors in \cite{HvR24}.
In contrast, showing quasi-optimal complexity for an \emph{adaptive
algorithm} using anisotropic wavelets is the topic of the present article.
Therefore, we shall  investigate the compressibility of the operator matrix 
\(\Lbfm\) in anisotropic wavelet coordinates. 
It will turn out that \(\Lbfm\) is indeed \(s^{\star}\)-compressible for some
\(s^{\star} > \bar{s}\), 
where we note that \(\bar{s}\) is the best rate that can be achieved by
sparse-grid approximation for a general, smooth function \cite{BG04,GH11,GK00}.
This means in particular that an adaptive algorithm in the sense of 
\cite{CDD01,CDD02} can be realised 
which provides an approximation \(u_N\) which converges to \(u\) 
at the rate \(N^{-s}\) for any \(s < \bar{s}\).
Moreover, there are only \(\Ocal(N)\) operations required
for the algorithmic computation of \(u_N\), 
provided the matrix entries are known in advance.
The interesting question whether the 
approximate \(u\) can also numerically be realised in linear 
computing time and, especially,
whether the matrix entries can be calculated sufficiently fast,
is however not addressed in this article.

The rest of this article is organised as follows. 
In Section~\ref{sec:Wavelet_Bases}, 
we introduce anisotropic wavelet bases on the unit square and 
define associated function spaces.
Then, in Section~\ref{sec:compressible_matrices}, 
we state the problem formulation and the basic estimates for the discrete 
counterpart of the boundary integral operator under consideration. 
The compression scheme is then defined and analysed 
in Section~\ref{sec:compression_general}.
The results are generalised to manifolds in Section~\ref{sec:manifold}.

To simplify notation, we will write \(A \lesssim B\) if \(B\) is bounded by a
constant multiple of \(A\), where the constant does not depend on any parameters
\(A\) or \(B\) might depend on.
Similarly, we write \(A \gtrsim B\) if and only if \(B \lesssim A\),
and we write \(A \sim B\) if both \(A \lesssim B\) and \(B \lesssim A\).

\section{Wavelet Bases}
\label{sec:Wavelet_Bases}

In this section, we define the wavelet bases under consideration
and state their most important properties. 
Throughout the article, 
we assume that the scaling functions and the wavelets involved 
are piecewise polynomial and
compactly supported with support size depending on their levels.
In particular, we assume that \(\Diam \Supp \psi \sim 2^{-j}\) holds
for a one-dimensional wavelet \(\psi\) on level \(j\). Such wavelets 
are, in general, biorthogonal. 
Related
bases were first constructed in \cite{CDF92} and, later, this 
construction was also carried out on the interval in \cite{DKU99}.

\subsection{Single-Scale Bases}

Consider a sequence of nested, finite-dimensional, and asymptotically dense 
function spaces
\begin{equation*}
	V_{j_0} \subset V_{j_0 +1} \subset \dots \subset V_{j-1} \subset V_j 
	\subset V_{j+1} \subset \dots \subset V,
\end{equation*}
which are used to discretise functions
on the unit interval \([0, 1]\).
We assume that the basis functions can be expressed as
shifts and dilations of a set of finitely many
compactly supported scaling functions \(\phi\). 
Moreover, for a suitable index set \(\Delta_j\), we assume that
\begin{equation*}
	\Phi_j \isdef \big\{ \phi_{\lambda} : \lambda \in \Delta_j \big\}
\end{equation*}
is a Riesz basis of \(V_j\) uniformly in \(j\),
i.e., there holds
\begin{equation*}
	\left\|\sum_{\lambda \in \Delta_{j}} c_{\lambda} \phi_{\lambda}
	\right\|_{V}
	\sim \sum_{\lambda \in \Delta_{j}} \norms[\big]{c_{\lambda}}^2
\end{equation*}
independent of \(j\).
Additionally, we assume that
\(\Diam\Supp \phi_{\lambda}\sim 2^{-j}\). 
Here, the index \(\lambda = (j,k)\) contains information 
on the scale \(j\) and the location \(k\). 
In the easiest case, one can think of \(\phi\) 
as the constant function \(1\), and of \(\phi_{\lambda}\) as a 
properly scaled, dyadic indicator function. If \(V = L^2([0,1])\), for example,
we denote \(\phi_{\lambda} = 2^{\nicefrac{j}{2}} \mathds{1}_{[2^{-j}k, 
\,2^{-j}(k+1)]}\).

We say that the spaces \(V_j\) have the approximation order \(d\) 
if they contain locally all polynomials up to the order \(d\). 
Moreover, \(V_j\) are said to have the regularity up to
\(\gamma\isdef \sup \{ s \in \Rbbb : V_j \subset H^s([0, 1])\}\).

\subsection{Multiscale Bases}

As the spaces \(V_j\) are nested, we may write
\begin{equation}
	\label{eq:complement_decomposition}
	V_j = V_{j-1} \oplus W_j
\end{equation}
for a suitable complement or difference space \(W_j\).
One can show that, if the scaling function \(\phi\) 
generates a shift-invariant space, cf.\ \cite{CDF92,DKU99,DS98},
there are finitely many mother wavelet functions such that
\begin{equation*}
	\Psi_j \isdef \big\{ \psi_{\lambda} : \lambda \in \nabla_j \big\}
\end{equation*}
is a basis set of \(W_j\). 
Also herein, all mother wavelets \(\psi\) are piecewise polynomial of the same
order as \(\phi\), \(\nabla_j\) is a suitable index set and 
\(\psi_{\lambda}\) is a properly scaled and 
translated copy of a mother wavelet \(\psi\). 

From \eqref{eq:complement_decomposition}, we recursively 
obtain the multiscale decomposition
\begin{equation*}
	V_j = V_{j_0} \oplus W_{j_0+1} \oplus \dots \oplus W_j,
\end{equation*}
provided that \(0 \leq j_0 < j\). 
As the function spaces \(V_j\) are asymptotically dense, 
meaning that \(V = \overline{\bigcup_{j \geq j_0}V_j}\),
we conclude that the span of
\begin{equation}
	\Psi \isdef \Phi_{j_0} \cup \bigcup_{j = j_0+1}^\infty \Psi_j
	\label{eq:multiscale_basis}
\end{equation}
is dense in \(V\). 

Moreover, if we follow the construction of \cite{CDF92,DKU99,DS98}, the set
\(\Psi\) forms a Riesz basis of \(V\), meaning that
\begin{equation*}
	\left\| \sum_{\lambda \in \nabla} c_{\lambda} \psi_{\lambda}
	\right\|_{V}^2
	\sim \sum_{\lambda \in \nabla} \big|c_{\lambda}\big|^2,
	\qquad
	\nabla \isdef \Delta_{j_0} \cup \bigcup_{j=j_0+1}^{\infty} \nabla_j.
\end{equation*}
Therefore, there exists a unique basis \(\tilde{\Psi} \isdef
\{\tilde{\psi}_{\lambda} : \lambda \in \nabla\}\) of the dual space \(V'\)
which is biorthogonal to \(\Psi\), that is
\begin{equation}
	\label{eq:biorthogonality}
	\big\langle \tilde{\psi}_{\lambda'}, \psi_{\lambda} \big\rangle =
	\delta_{\lambda,\lambda'},
	\qquad \lambda,\lambda' \in \nabla,
\end{equation}
where \(\groupa{\cdot, \cdot}\) denotes the duality pairing.
In particular, \(\tilde{\Psi}\) is also a Riesz basis of \(V'\).

Following the construction of \cite{CDF92,DKU99,DS98},
the spaces \(\tilde{V}_{j_0} \isdef \{\tilde{\psi}_{\lambda} : \lambda \in
\Delta_{j_0}\}\) and \(\tilde{W}_j \isdef \{\tilde{\psi}_{\lambda} : \lambda \in
\nabla_j\}\) for \(j \geq j_0 + 1\) form a multiresolution analysis
\begin{equation*}
	\tilde{V}_{j_0} \subset \tilde{V}_{j_0+1} \subset \dots \subset
	\tilde{V}_{j-1} \subset \tilde{V}_{j} \subset \tilde{V}_{j+1} \dots
	\subset V'
\end{equation*}
in \(V'\), where for \(j \geq j_0 + 1\)
\begin{equation*}
	\tilde{V}_j \isdef \tilde{V}_{j_0} \oplus \tilde{W}_{j_0+1} \oplus \dots
	\tilde{W}_{j}.
\end{equation*}
The spaces \(\tilde{V}_j\isdef\{\tilde\phi_{\lambda}: 
\lambda\in\Delta_j\}\) admit the regularity \(\tilde{\gamma} > 0\) 
and the approximation order \(\tilde{d}\). 
Especially, because of \eqref{eq:biorthogonality}, 
we can conclude that \(\langle p, \, \psi \rangle = 0\) 
for any polynomial \(p\) up to the order \(\tilde{d}\), 
meaning that the wavelet \(\psi\) admits \(\tilde{d}\) vanishing moments.

A consequence of the vanishing moments is 
the cancellation property,
which means that
\begin{equation}
	\label{eq:cancellation_property_univariate}
	\big| \langle \psi_{\lambda}, u \rangle \big|
	\lesssim 2^{-(\tilde{d} + \frac{1}{2})j} \norms{u}_{W^{\tilde{d},
	\infty}(\Supp \psi_{\lambda})}, 
	\qquad u \in W^{\tilde{d}, \infty}(\Supp \psi_{\lambda})
\end{equation}
for \(\lambda \in \nabla_{j}\),
cf.\ \cite{Dah97,Sch98}.

\subsection{Tensor Product Bases}

Up to now, we have considered function spaces on the unit interval.
When it comes to \(n > 1\) dimensions, there are basically two approaches.
The first one is to use isotropic wavelets, which on level \(j\) are given 
as tensor products of wavelet functions on a level \(j\) and scaling 
functions on roughly the same level.
Another approach is to use the tensor product of one-dimensional 
wavelets on all the different levels, which will be the interest of this article.
In the context of boundary integral equations, they were first considered 
in \cite{GOS99} as bases of sparse tensor product spaces. 

In order to keep the construction simple, we fix 
the number of dimensions to \(n = 2\), but we would like to emphasise 
that an extension to higher dimensions is straightforward. To this
end, let \(\square \isdef [0, 1]^2\). For a given multiindex \(\jbfm 
\isdef (\jx, \jy) \geq \jbfm_0+\onebfm\), we define the index set 
\(\nablabfm_{\jbfm} \isdef \nabla_{\jx} \times \nabla_{\jy}\), and for 
\(\lambdabfm = (\lambdax, \lambday) \in \nablabfm_{\jbfm}\), 
we define the wavelet function
\begin{align*}
	\psi_{\lambdabfm}(\xvec) 
	\isdef \big(\psi_{\lambdax}\otimes \psi_{\lambday} \big) (\xvec)
	= \psi_{\lambdax}(x) \psi_{\lambday}(y).
\end{align*}
All such wavelets on a level \(\jbfm\) span the corresponding 
complement space, i.e.,
\begin{equation*}
	\Psi_{\jbfm} \isdef \big\{ \psi_{\lambdabfm} : \lambdabfm \in\nablabfm_{\jbfm} \big\}
\end{equation*}
is a basis set of \(W_{\jx} \otimes W_{\jy}\). If we also re-define 
\(\psi_{\lambda} \isdef \phi_{\lambda}\) for \(\lambda \in\nabla_{j_0}\)
and extend the above definition to any multiindex \(\jbfm \geq \jbfm_0\), 
then we deduce out of \eqref{eq:multiscale_basis} that the span of
\begin{equation*}
	\Psi \isdef \{ \psi_{\lambdabfm} : \lambdabfm \in\nablabfm \} 
	= \bigcup_{\jbfm \geq \jbfm_0} \Psi_{\jbfm},
	\qquad 
	\nablabfm \isdef\bigcup_{\jbfm \geq \jbfm_0} \nablabfm_{\jbfm},
\end{equation*}
is dense in \(V \otimes V\) and forms a Riesz basis. For the sake of convenience, 
we shall denote in the following
\begin{equation*}
	|\lambdabfm|_1 \isdef |\jbfm|_1, \qquad
	|\lambdabfm|_\infty \isdef |\jbfm|_\infty.
\end{equation*}

By considering in complete analogy the tensor products of 
the dual wavelets \(\tilde{\psi}_{\lambda}\), we get a unique 
dual Riesz basis \(\tilde{\Psi} = \{\tilde{\psi}_{\lambdabfm} : \lambdabfm
\in\nablabfm\}\) of the dual space \(V' \times V'\), which satisfies
\begin{equation*}
	\big\langle \tilde{\psi}_{\lambdabfm'}, \, \psi_{\lambdabfm}\big\rangle
	= \delta_{\lambdabfm, \lambdabfm'},
\end{equation*}
and therefore \(\|\tilde{\psi}_{\lambdabfm}\|_{V'\times V'}\sim 1\).
In particular, if \(V = L^{2}([0, 1])\), \(\Psi\) and
\(\tilde{\Psi}\) constitute biorthogonal Riesz bases for \(L^{2}(\square)\).

\subsection{Function Spaces}

Assuming \(L^{2}(\square)\)-normalisation, a wavelet \(\psi_{\lambdabfm}\), 
with \(\lambdabfm \in\nablabfm_{\jbfm}\) can be written as
\begin{equation}
	\label{eq:scaling_factor_for_L2_wavelet}
	\psi_{\lambdabfm}(\xvec) = 2^{\frac{1}{2}|\jbfm|_1} 
	\psi(2^{\jx}x -\kx) \psi(2^{\jy}y -\ky), \quad \kbfm\in \Nbbb_0^2.
\end{equation}
Therefore, every function \(u \in L^2(\square)\) 
admits a unique expansion
\begin{equation*}
	u = \sum_{\lambdabfm \in\nablabfm}
	u_{\lambdabfm} \psi_{\lambdabfm},
	\qquad u_{\lambdabfm} = 
	\big\langle \tilde{\psi}_{\lambdabfm}, u \big\rangle.
\end{equation*}

We intend next to characterise certain function spaces. 
As already stated, the biorthogonality \eqref{eq:biorthogonality} 
implies that the integral of a wavelet
against any polynomial up to order \(\tilde{d}\) vanishes.
By considering a Taylor expansion one can conclude that, for any sufficiently
regular function \(u\), there hold the \emph{cancellation properties}
\begin{align}
	\label{eq:cancellation_property_elementary}
	\big| \langle u, \psi_{\lambdabfm} \rangle_{\square} \big| 
	&\lesssim 2^{-\frac{1}{2} |\lambdabfm|_1} 
	\min\left\{ 2^{-\tilde{d} |\lambdabfm|_1} |u|_{W^{2
		\tilde{d},\infty}(\Supp \psi_{\lambdabfm})},\ 
		2^{-\tilde{d}|\lambdabfm|_{\infty}}
	|u|_{W^{\tilde{d},\infty}(\Supp \psi_{\lambdabfm})} \right\}, \\
	\label{eq:cancellation_property_elementary_dual}
	\big| \langle \tilde{\psi}_{\lambdabfm}, u \rangle_{\square} \big| 
	&\lesssim 2^{-\frac{1}{2} |\lambdabfm|_1} 
	\min\left\{ 2^{-d |\lambdabfm|_1} |u|_{W^{2
		d,\infty}(\Supp \tilde{\psi}_{\lambdabfm})},\ 
		2^{-d|\lambdabfm|_{\infty}}
	|u|_{W^{d,\infty}(\Supp \tilde{\psi}_{\lambdabfm})} \right\},
\end{align}
cf.\ \cite{HvR24,Rei08}.
Therefore, if \(u\) is sufficiently regular, the coefficients of the wavelet
expansion of \(u\) decay exponentially in the level of the wavelet.

Furthermore, the decay behaviour of coefficients can be used to 
characterise the norm of the corresponding function 
with respect to certain function spaces. 
If we define \(H^s(\square)\) to be the Sobolev space of
regularity \(- \tilde{\gamma} < s < \gamma\) with the 
boundary conditions imposed by the wavelet basis \(\Psi\), and 
\(H^{-s}(\square) \isdef (H^{s}(\square))'\) as the Sobolev space of
regularity \(-s\) with the boundary conditions imposed by \(\tilde{\Psi}\),
there holds
\begin{equation}
	\label{eq:norm_equivalence}
	\begin{aligned}
		\| u \|_{H^s(\square)}^2 &\sim \sum_{\lambdabfm \in\nablabfm} 2^{2s |\lambdabfm|_\infty}
		\big|\langle\tilde{\psi}_{\lambdabfm},u\rangle\big|^2,\\
		\| u \|_{H^{-s}(\square)}^2 &\sim \sum_{\lambdabfm \in\nablabfm} 2^{-2s |\lambdabfm|_\infty}
		\big|\langle u,\psi_{\lambdabfm}\rangle\big|^2,
	\end{aligned}
	\qquad\qquad - \tilde{\gamma} < s < \gamma, \\
\end{equation}
cf.\ \cite{Dah97,GK00,GOS99,HvR24,HvR24b,Sch98,SUV21} and the references therein.
In particular, a function \(u\) is contained in an isotropic function space
if the coefficients of its anisotropic wavelet expansion decay sufficiently fast.
For details on the realisation of boundary conditions, we refer
to \cite{DS98,DS99,HS06}.

\begin{remark}
	The norm equivalence \eqref{eq:norm_equivalence} 
	implies that \(\Psi\) and \(\tilde{\Psi}\) form Riesz bases 
	for a whole range of Sobolev spaces on the unit square.
\end{remark}

The norm equivalence \eqref{eq:norm_equivalence} was established in 
\cite{GOS99} for a smaller range of parameters, 
and there was also shown that
the coefficients with respect to a tensor product basis characterise
the dominating mixed regularity of a function \(u\).
Shortly afterwards, in \cite{GK00}, \eqref{eq:norm_equivalence} 
was posed on the \(n\)-dimensional torus.
In the same article, function spaces of 
\emph{hybrid regularity} were introduced. 
Roughly speaking, these spaces contain all 
functions \(u\) which admit mixed derivatives in an isotropic Sobolev space.

\begin{definition}
	\label{def:Griebel_Knapek_space}
	We define \(\Hfrak^{q,\sbfm}(\square)\) for \(q \geq 0\) and \(\sbfm = (\sx,
	\sy)\) as the space 
	\begin{equation*}
		\Hfrak^{q,\sbfm}(\square) = H^{q+\sx}([0,1]) \otimes H^{\sy}([0,1])
		\cap H^{\sx}([0,1]) \otimes H^{q+\sy}([0,1]).
	\end{equation*}
	For \(q < 0\), we define the space \(\Hfrak^{q,\sbfm}(\square)\) via the duality 
	\(\Hfrak^{q,\sbfm}(\square) = \big(\Hfrak^{-q,-\sbfm}(\square)\big)'\).
	Moreover, we define \(\Hfrak^{q,s}(\square) \isdef
	\Hfrak^{q,(s,s)}(\square)\).
\end{definition}

As proposed in \cite{GK00}, also the spaces \(\Hfrak^{q,s}(\square)\) can be
characterised by the decay behaviour of wavelets coefficients.
By using the arguments of \cite[Theorem 3.3]{HvR24b}, we immediately see that the results
of \cite{GK00,HvR24b} can be extended to the spaces
\(\Hfrak^{q,\sbfm}(\square)\) as well.
\begin{theorem}
	\label{thm:Griebel_Knapek_norm_equivalence}
	There holds
	\begin{align}
		\label{eq:norm_equivalence_GK}
		\|u\|_{\Hfrak^{q,\sbfm}(\square)}^2 &\sim \sum_{\lambdabfm \in\nablabfm}
		2^{2q|\lambdabfm|_\infty + 2\sbfm \cdot \lambdabfm}
		\big|\langle\tilde{\psi}_{\lambdabfm},u\rangle\big|^2,
		\quad -\tilde{\gammabfm} < \sbfm, \qbfm+\sbfm < \gammabfm,\\
		\label{eq:norm_equivalence_GK_dual}
		\|u\|_{\Hfrak^{q,\sbfm}(\square)}^2 &\sim \sum_{\lambdabfm \in\nablabfm}
		2^{2q|\lambdabfm|_\infty + 2\sbfm \cdot \lambdabfm}
		\big|\langle u,\psi_{\lambdabfm}\rangle\big|^2,
		\quad -\gammabfm < \sbfm, \qbfm+\sbfm < \tilde{\gammabfm}.
	\end{align}
	Herein, for \(\lambdabfm = \big((\jx,\kx), (\jy,\ky)\big)\), we denote
	\(\sbfm \cdot \lambdabfm \isdef \sx \jx + \sy \jy\),
	and the inequalities are understood component-wise.
\end{theorem}

\begin{remark}
	The one-sided upper bound in \eqref{eq:norm_equivalence_GK} 
	can be extended
	for \(-\tilde{\dbfm} < \sbfm, \qbfm+\sbfm < \gammabfm\), whereas the one-sided lower bound can
	be extended to \(-\tilde{\gammabfm} < \sbfm, \qbfm+\sbfm < \dbfm\). Similarly, we can also
	extend the upper and lower bounds in \eqref{eq:norm_equivalence_GK_dual} up
	to \(-\dbfm\) and \(\tilde{\dbfm}\), respectively. See e.g.~\cite{Sch98} for the details.
\end{remark}

We note that the spaces \(\Hfrak^{q,\sbfm}(\square)\) are 
Hilbert spaces with the inner product
\begin{equation*}
	\langle u,v\rangle_{\Hfrak^{q,\sbfm}(\square)} \isdef \sum_{\lambdabfm \in\nablabfm}
	2^{2q|\lambdabfm|_\infty + 2\sbfm \cdot \lambdabfm}
	\langle\tilde{\psi}_{\lambdabfm}, u\rangle \langle\tilde{\psi}_{\lambdabfm}, v\rangle,
\end{equation*}
and that we have the Gelfand triples \(\Hfrak^{q,\sbfm}(\square)\hookrightarrow H^q(\square)
\hookrightarrow \Hfrak^{q,-\sbfm}(\square)\), provided that \(\sbfm > \zerobfm\).
Due to the tensor product structure of \(\psi_{\lambdabfm}\),
we also have \(\psi_{\lambdabfm} \in \Hfrak^{0,s}(\square)\) 
for any \(s < \gamma\) and it is known that \(\Hfrak^{0,\gamma}(\square)\) 
is continuously embedded into \(\Hfrak^{q,s}(\square)\) if either
\(q\geq0\) and \(s<\gamma-q\),
or if \(q<0\) and \(s \leq \gamma\), cf.\
\cite[Lemma 2.8]{BHW23}.

\section{Problem Formulation}
\label{sec:compressible_matrices}

We are now in the position to introduce the problem under consideration. 
In what follows, we intend to adaptively solve an equation of the kind
\begin{equation}
	\label{eq:BOE}
	\Lcal u = g \in H^{-q}(\square),
\end{equation}
where \(\Lcal\) is a  bounded and
continuously invertible (possibly with appropriate restrictions on the energy space)
integral operator 
with kernel \(\kappa\colon \square \times \square \to \Rbbb\).
In particular, this means that there holds 
the representation
\begin{equation*}
	\big(\Lcal u\big) (\xvec) = \int_{\square} \kappa(\xvec, \xvec')
	u(\xvec') \dint \xvec', \qquad \xvec \in \square,
\end{equation*}
provided that \(u\) is sufficiently smooth.
This kernel is assumed to be asymptotically smooth of the 
order \(2q \in \Rbbb\), 
meaning that \(\kappa\) is smooth apart from the diagonal \(\{(\xvec,\xvec') \in
\square \times \square : \xvec = \xvec'\}\),
and its derivatives are asymptotically bounded by
\begin{equation}
	\label{eq:asymptotic_smooth_decay}
	\big| \partial_{\xvec}^{\alphabfm} \partial_{\xvec'}^{\alphabfm'} 
	\kappa(\xvec,\xvec') \big|
	\lesssim \|\xvec - \xvec'\|^{-(2 + 2q + |\alphabfm| + |\alphabfm'|)}
\end{equation}
whenever \(2 + 2q + |\alphabfm| + |\alphabfm'| > 0\).
Such kernel functions arise, for instance, by applying an 
integral formulation to Dirichlet or Neumann screen problems 
in \(\mathbb{R}^3\). The extension of \eqref{eq:BOE} to general 
manifolds is considered later in Section~\ref{sec:manifold}.

\subsection{Approximation Spaces}

The ultimate goal is the solution of the integral equation
\eqref{eq:BOE} with asymptotically optimal complexity. 
This means in particular that the approximate solution 
converges for a fixed, maximal number of arithmetic operations
\(N\) to the original solution at the convergence rate \(N^{-s}\), 
obtainable if full knowledge on the function \(u\) was provided.

This question falls under the regime of 
\emph{nonlinear approximation} \cite{DeV98}. 
The first requirement is therefore to know for which functions \(u\), 
the error behaviour of the best \(N\)-term approximation,
\begin{equation*}
	E_N(u) \isdef \inf\left\{ \|u - v_N\|_{H^q(\square)} 
		: v_N = \sum_{\lambdabfm \in \Lambda}
	c_{\lambdabfm} \psi_{\lambdabfm}, \ |\Lambda| \leq N\right\},
\end{equation*}
decays at a given rate \(N^{-s}\).
Of course, this rate depends on the regularity of \(u\).
As proven in \cite{HvR24b,Nit04,Nit06}, 
\(u\) must admit a sufficiently high, hybrid Besov regularity.  
In particular, there is a guarantee that all functions
\(u \in \Bfrak^{q,s,\tau}_{\tau}(\square)\), with \(\frac{1}{\tau} = s +
\frac{1}{2}\) can be approximated with \(N\) terms at the rate \(N^{-s}\).
In other terms, whenever the coefficients of the wavelet expansion of
\(u\) decay sufficiently fast to guarantee that
\begin{equation}
	\label{eq:Besov_tau_criterion}
	|u|_{\Bfrak^{q,s,\tau}_{\tau}(\square)} \isdef
	\left[
		\sum_{\lambdabfm \in \nablabfm} 2^{\tau q |\lambdabfm|_{\infty}}
		\big|\big\langle \tilde{\psi}_{\lambdabfm}, u\rangle\big|^{\tau} 
	\right]^{\nicefrac{1}{\tau}}
	< \infty,
\end{equation}
then \(u\) is approximable at the rate \(N^{-s}\).

\subsection{\(s^{\star}\)-Compressibility}

For the sake of simplicity, let us rescale our wavelet functions.
As we consider a linear operator
\(\Lcal\colon H^q(\square) \to H^{-q}(\square)\),
it is meaningful to directly consider a Riesz basis for \(H^q(\square)\),
as it was done, for example, in \cite{CDD01,CDD02,DHS07,GS06,Ste04}.
Therefore, for the remainder of this article,
we assume that the wavelets \(\psi_{\lambdabfm}\) are scaled
such that they form a Riesz basis of \(H^q(\square)\).
In other terms, we denote
\begin{equation}
	\label{eq:rescaling}
	\psi_{\lambdabfm} \isdef 2^{-q|\lambdabfm|_\infty} \psi_{\lambdax}\otimes \psi_{\lambday}, \quad
	\tilde{\psi}_{\lambdabfm} \isdef 2^{q|\lambdabfm|_\infty} \tilde{\psi}_{\lambdax}
	\otimes \tilde{\psi}_{\lambday},
\end{equation}
where the one-dimensional wavelets \(\psi_{\lambdax}\), \(\psi_{\lambday}\) and 
their duals are normalised in \(L^2([0,1])\).
If the wavelets \(\psi_{\lambdabfm}\) admit suitable regularity to discretise
\(H^{q}(\square)\),
the norm equivalence \eqref{eq:norm_equivalence} immediately tells us that
the sets
\begin{equation*}
	\Psi \isdef \{ \psi_{\lambdabfm} : \lambdabfm \in \nablabfm\}, \quad
	\tilde{\Psi}\isdef \{ \tilde{\psi}_{\lambdabfm} : \lambdabfm \in \nablabfm\}
\end{equation*}
are Riesz bases of \(H^q(\square)\) and \(H^{-q}(\square)\), 
respectively, and that they are biorthogonal dual bases provided 
that \(-\tilde{\gamma} < q < \gamma\).

Let us now discretise the operator equation \eqref{eq:BOE}
with the help of the basis \(\Psi\) and its dual \(\tilde{\Psi}\). As 
\begin{equation*}
	g = \sum_{\lambdabfm \in \nablabfm} g_{\lambdabfm} \tilde{\psi}_{\lambdabfm},
	\quad g_{\lambdabfm} = \langle g, \psi_{\lambdabfm} \rangle_\square,
\end{equation*}
we may also expand \(\Lcal u\) with respect to \(\tilde{\Psi}\),
and use the biorthogonality to conclude that all the respective
coefficients must be equal. Therefore, any solution of \eqref{eq:BOE} 
is also a solution of the variational formulation
\begin{equation}
	\label{eq:variational_formulation}
	\text{find } u \in H^q(\square) \text{ such that } 
	\langle \Lcal u, \, 
	\psi_{\lambdabfm} \rangle_{\square}
	= \langle g, \, \psi_{\lambdabfm} \rangle_{\square}
	\text{ for any } \lambdabfm \in \nablabfm.
\end{equation}

As \(\Psi\) is a Riesz basis for \(H^q(\square)\)
and we are looking for \(u \in H^q(\square)\),
we might as well look for the coefficient vector
of the unknown function \(u\) with respect to the Riesz basis \(\Psi\), 
which we denote by
\(\ubfm = [u_{\lambdabfm}]_{\lambdabfm} \in \ell^2(\nablabfm)\).
The variational formulation \eqref{eq:variational_formulation} 
can therefore be written as a bi-infinite linear system 
of equations \(\Lbfm \ubfm = \gbfm\), where
\begin{align*}
	\Lbfm = \big[ 
		\langle \Lcal \psi_{\lambdabfm'}, \psi_{\lambdabfm} \rangle_{\square}
	\big]_{\lambdabfm, \lambdabfm' \in \nablabfm},
	\quad \gbfm = \big[ g_{\lambdabfm}	\big]_{\lambdabfm \in \nablabfm} \in \ell^2(\nablabfm).
\end{align*}
Hence, the solution of the operator equation is 
equivalent to a discrete problem in \(\ell^2(\nablabfm)\),
where the discrete operator \(\Lbfm\) is well-conditioned 
since \(\Psi\) is a Riesz basis and the operator \(\Lcal\) is
continuous and elliptic.
Therefore, for \(\vbfm \in \ell^2(\nablabfm)\) and \(v = 
\sum_{\lambdabfm \in\nablabfm} v_{\lambdabfm} \psi_{\lambdabfm}\),
there holds
\begin{equation*}
	\|\gbfm - \Lbfm \vbfm\|_{\ell^2(\nablabfm)} \sim \|g - \Lcal v
	\|_{H^{-q}(\square)}
	\sim \| u - v\|_{H^{q}(\square)} \sim \|\ubfm - \vbfm\|_{\ell^2(\nablabfm)},
\end{equation*}
that is, the error between the solution \(u\) and a function \(v\) is 
equivalent to the residual of the discrete, bi-infinite linear system 
of equations.

Our next question is:
Under which conditions is it possible to
approximate the unknown coefficient vector \(\ubfm\)
up to an error \(N^{-s}\), requiring at most
a bounded multiple of \(N\) operations?
This problem was addressed first in \cite{CDD01,CDD02} and 
the key ingredient is that the system matrix \(\Lbfm\) admits a high 
degree of compressibility, namely \emph{\(s^\star\)-compressibility}.

\begin{definition}
	\label{def:s_star_compressibility}
	We say that a bi-infinite matrix \(\Lbfm\) is \(s^\star\)-compressible, 
	if there exist two non-negative, summable sequences \( (\alpha_r)_r\) 
	and \( (\beta_r)_r\) such that for any \(r \in \Nbbb\), 
	there exists another bi-infinite matrix \(\Lbfm_{r}\) satisfying 
	\(\|\Lbfm - \Lbfm_{r}\|_{\ell^2(\nablabfm)} \lesssim \beta_{r} 2^{-sr} \) 
	for any \(s < s^\star\), while each row and each column 
	contains at most a bounded multiple of \(\alpha_r 2^r\) entries.
\end{definition}

\subsection{Basic definitions}

For the sake of convenience, let us denote
\begin{equation*}
	\mx(\jbfm,\jbfm') \isdef \min\{\jx,\jx'\},\quad
	\my(\jbfm,\jbfm') \isdef \min\{\jy,\jy'\},\quad
	\minm(\jbfm,\jbfm') \isdef \min\{\jx,\jy,\jx',\jy'\},
\end{equation*}
and similarly
\begin{equation*}
	\Mx(\jbfm,\jbfm') \isdef \max\{\jx,\jx'\},\quad
	\My(\jbfm,\jbfm') \isdef \max\{\jy,\jy'\},\quad
	\maxm(\jbfm,\jbfm') \isdef \max\{\jx,\jy,\jx',\jy'\}.
\end{equation*}
Additionally, we write \(\minm(\lambdabfm,\lambdabfm')\) and
similarly for \(\mx\) and \(\my\) to denote the minimal 
levels of a pair of multiindices and likewise for the maximum. 
Moreover, the support of the wavelet \(\psi_{\lambdabfm}\) is denoted by
\(\Omega_{\lambdabfm} \isdef \Supp \psi_{\lambdabfm}\),
and for \(\lambdabfm,\lambdabfm'\in\nablabfm\),
we define the distance between between the supports of the associated wavelets
\(\psi_{\lambdabfm}\) and \(\psi_{\lambdabfm'}\) by
\begin{equation*}
	\delta(\lambdabfm,\lambdabfm')\isdef
	\Dist(\Omega_{\lambdabfm},\Omega_{\lambdabfm'}).
\end{equation*}
In analogy, we denote \(\Omega_{\lambdax}
\isdef \Supp\psi_{\lambdax}\),
\begin{equation*}
	\deltax(\lambdabfm,\lambdabfm')\isdef
	\Dist(\Omega_{\lambdax}, \Omega_{\lambdax'}),
\end{equation*}
and likewise in the coordinate direction \({\rm y}\).

We shall also define the points where a wavelet is not smooth as 
\(\Omega_{\lambdabfm}^{\sigma} \isdef\Singsupp \psi_{\lambdabfm}\), i.e., 
a wavelet \(\psi_{\lambdabfm}\) is smooth
on \(\square \setminus \Omega_{\lambdabfm}^{\sigma}\). 
With this, we define for \(\lambdabfm \in \nabla_{\jbfm}\) 
and \(\lambdabfm' \in \nablabfm_{\jbfm'}\),
\begin{equation*}
	\sigmax(\lambdabfm,\lambdabfm')\isdef
	\begin{cases}
		\Dist(\Omega_{\lambdax},\Omega_{\lambdax'}^{\sigma}),
		&\text{if } \jx \geq \jx',\\
		\Dist(\Omega_{\lambdax}^{\sigma},\Omega_{\lambdax'}),
		&\text{if } \jx' \geq \jx.
	\end{cases}
\end{equation*}
Likewise, we define \(\sigmay(\lambdabfm,\lambdabfm')\).
Note that, in both coordinate directions, the quantities 
\(\sigmax(\lambdabfm,\lambdabfm')\) and
\(\sigmay(\lambdabfm,\lambdabfm')\) always exist.

\subsection{Wavelet Estimates}

First, we remark that we can immediately deduce 
the two estimates
\begin{align}
	\label{eq:cancellation_property}
	\big|\langle\Lcal\psi_{\lambdabfm'},\psi_{\lambdabfm}\rangle\big|
	&\lesssim 2^{-(\tilde{d}+\frac{1}{2})(|\lambdabfm|_1+|\lambdabfm'|_1)}
	2^{-q(|\lambdabfm|_\infty + |\lambdabfm'|_\infty)}
	\delta(\lambdabfm,\lambdabfm')^{-(2+2q+4\tilde{d})},\\
	\label{eq:cancellation_property_two_vanmom}
	\big|\langle\Lcal\psi_{\lambdabfm'},\psi_{\lambdabfm}\rangle\big|
	&\lesssim 2^{-\frac{1}{2}(|\lambdabfm|_1+|\lambdabfm'|_1)}
	2^{-\tilde{d}(j^{(1)}+j^{(2)})}
	2^{-q(|\lambdabfm|_\infty+|\lambdabfm'|_\infty)}
	\delta(\lambdabfm,\lambdabfm')^{-(2+2q+2\tilde{d})}.
\end{align}
Indeed, by adopting the scaling \eqref{eq:rescaling},
the estimates
\eqref{eq:cancellation_property} and \eqref{eq:cancellation_property_two_vanmom}
follow immediately from \cite[Theorem 3]{HvR24} and \cite[Theorem 2.1.9]{Rei08},
respectively.
In \eqref{eq:cancellation_property_two_vanmom},
\(j^{(1)}\) and \(j^{(2)}\) may denote any pair of the four levels 
\(\{\jx,\jy,\jx',\jy'\}\) involved, 
where a level can be chosen as many times as it occurs in this set.
Obviously, the best choice is taking the two largest entries.

The estimates \eqref{eq:cancellation_property} and
\eqref{eq:cancellation_property_two_vanmom}, however, do 
only hold if \(\delta(\lambdabfm,\lambdabfm')>0\), which is quite a 
strong requirement due to the anisotropic structure of our wavelets.
Therefore, we also need estimates on wavelet entries for which the
supports of the wavelets involved overlap. 
As we will see, in order to discard a matrix entry, 
we  have to require that one wavelet is located 
sufficiently far away from the singular support of the other wavelet 
in at least one coordinate direction.

First, we want to derive an estimate for a wavelet
pair, for which one wavelet is located at the, possibly 
very long, face of another wavelet, cf.\ Figure \ref{fig:support_picture}.

\begin{figure}[h]
	\centering
	\scalebox{0.5}{\begin{tikzpicture}
	\begin{LARGE}
	\draw[fill=cmblue!20!white, very thick] (-8,0) -- (8,0) -- (8,2) -- (-8,2) -- cycle;
	\draw[fill=cmred!20!white, very thick] (3,4) -- (4,4) -- (4,8) -- (3,8) -- cycle;
	\draw[<->] (3.5, 2) -- (3.5, 4);

	\node at (4, -0.5) {\(2^{-\jx}\)};
	\node at (8.7, 1) {\(2^{-\jy}\)};
	
	\node at (3.7, 8.5) {\(2^{-\jx'}\)};
	\node at (4.7, 6) {\(2^{-\jy'}\)};

	\node at (4.7, 3.0) {\(\deltay(\lambdabfm,\lambdabfm')\)};

\end{LARGE}

\end{tikzpicture}}
	\caption{Illustration of the situation in Lemma \ref{lm:long_face_estimate}.}
	\label{fig:support_picture}
\end{figure}
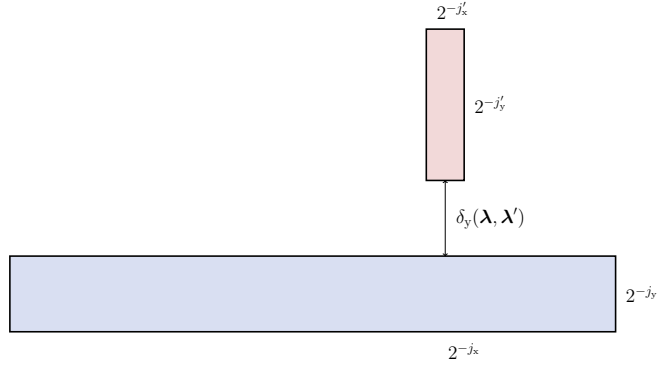

\begin{lemma}
	\label{lm:long_face_estimate}
	Assume that \(\lambdabfm \in\nablabfm_{\jbfm}\), 
	\(\lambdabfm' \in\nablabfm_{\jbfm'}\) are such that 
	\(\mx(\jbfm,\jbfm') \leq\my(\jbfm,\jbfm')\), and consider a 
	wavelet pair which satisfies \(\deltay(\lambdabfm,\lambdabfm')>0\).
	Then, we have
	\begin{align*}
		\big|\langle\Lcal\psi_{\lambdabfm'}, \psi_{\lambdabfm}\rangle\big|
		\lesssim 2^{-\frac{1}{2}(\jy+\jy' +|\jx-\jx'|)}
		2^{-(\tilde{d}+q)(|\jbfm|_{\infty} + |\jbfm'|_{\infty})} 
		\deltay(\lambdabfm,\lambdabfm')^{-(1+2q+2\tilde{d})}.
	\end{align*}
\end{lemma}
\begin{proof}
	Let us assume without loss of generality that \(\jx \leq\jx'\),
	which implies, together with \(\mx(\jbfm,\jbfm') \leq \my(\jbfm,\jbfm')\),
	that \(|\jbfm|_{\infty} = \jy\).
	Then, by using the one-dimensional 
	cancellation property \eqref{eq:cancellation_property_univariate}, 
	together with the appropriate scaling of the wavelets, 
	we use the vanishing moments to derive the decay estimate
	\begin{align}
		\notag
		\big|\langle\Lcal\psi_{\lambdabfm'}, \psi_{\lambdabfm}\rangle\big|
		&\lesssim 2^{-\frac{1}{2}(\jy + \jy' + \jx')
		-(\tilde{d} + q)(|\jbfm|_\infty+|\jbfm'|_\infty)}\\
		\label{eq:W_dtilde_1}
		&\qquad \bigg|\int_{\Omega_{\lambdax}}
		\kappa\big(x, \cdot, \cdot, \cdot \big)
		\psi_{\lambdax} (x) \dint x 
		\bigg|_{W^{2\tilde{d},\infty}(\Omega_{\lambday} \times
				\Omega_{\lambdax'}
		\times \Omega_{\lambday'})}.
	\end{align}
	To bound \eqref{eq:W_dtilde_1}, we remark that
	\(\|\psi_{\lambdax}\|_{L^{\infty}(\Omega_{\lambdax})} \lesssim
	2^{\nicefrac{\jx}{2}}\).
	Together with the decay estimate on the kernel
	\eqref{eq:asymptotic_smooth_decay}, we have
	\begin{align}
		\notag
		\big|\langle\Lcal\psi_{\lambdabfm'}, \psi_{\lambdabfm}\rangle\big|
		&\lesssim 2^{-\frac{1}{2}(\jy + \jy'
			+ |\jx-\jx'|)
		-(\tilde{d} + q)(|\jbfm|_\infty+|\jbfm'|_\infty)}\\ 
		\label{eq:lemma_long_face_to_plug_in}
		&\qquad \esssup_{(y, x', y')\in\Omega_{\lambday}
		\times\Omega_{\lambdax'} \times\Omega_{\lambday'}}
		\int_{\Omega_{\lambdax}}\|\xvec - \xvec'\|^{-(2+2q+2\tilde{d})} \dint x.
	\end{align}

	In order to estimate \eqref{eq:lemma_long_face_to_plug_in},
	let us fix 
	\((y, x', y')\in \Omega_{\lambday} \times \Omega_{\lambdax'} \times\Omega_{\lambday'}\). 
	We then decompose the integral over \(\Omega_{\lambdax}\) 
	into a part near \(x'\) and a part further away from \(x'\), i.e.,
	\begin{align*}
		\int_{\Omega_{\lambdax}} \|\xvec - \xvec'\|^{-(2+2q+2\tilde{d})} \dint x
		&\lesssim \int_{0}^{\deltay(\lambdabfm,\lambdabfm')}
		\big(t^2 + |y-y'|^2\big)^{-(1+q+\tilde{d})} \dint t\\
		&\qquad + \int_{\deltay(\lambdabfm,\lambdabfm')}^{\infty}
		\big(t^2 + |y-y'|^2\big)^{-(1+q+\tilde{d})} \dint t.
	\end{align*}
	For the first integral, we can immediately conclude that
	\begin{align*}
		\int_{0}^{\deltay(\lambdabfm,\lambdabfm')}
		\big(t^2 + |y-y'|^2\big)^{-(1+q+\tilde{d})} \dint t
		&\leq \deltay(\lambdabfm,\lambdabfm')|y-y'|^{-(2+2q+2\tilde{d})}\\
		&\leq \deltay(\lambdabfm,\lambdabfm')^{-(1+2q+2\tilde{d})}.
	\end{align*}
	For the second integral, on the other hand, we have
	\begin{align*}
		\int_{\deltay(\lambdabfm,\lambdabfm')}^{\infty}
		\big(t^2 + |y-y'|^2\big)^{-(1+q+\tilde{d})} \dint t
		&\lesssim \int_{\deltay(\lambdabfm,\lambdabfm')}^{\infty}
		t^{-(2+2q+2\tilde{d})} \dint t\\
		&\lesssim \deltay(\lambdabfm,\lambdabfm')^{-(1+2q+2\tilde{d})}.
	\end{align*}
	Hence, by plugging this into \eqref{eq:lemma_long_face_to_plug_in}, 
	we get the desired estimate.
\end{proof}

By an interchange of the coordinate directions, we immediately deduce the
following corollary.
\begin{corollary}
	If \(\lambdabfm \in\nablabfm_\jbfm, \lambdabfm' \in\nablabfm_{\jbfm'}\) and
	\(\my(\jbfm,\jbfm') \leq \mx(\jbfm,\jbfm')\), then for any pair of wavelets
	satisfying \(\deltax(\lambdabfm,\lambdabfm') > 0\), there holds
	\begin{equation*}
		\big|\big\langle \Lcal\psi_{\lambdabfm'},
		\psi_{\lambdabfm}\big\rangle\big|
		\lesssim 2^{-\frac{1}{2}(\jx + \jx' + |\jy-\jy'|)}
		2^{-(\tilde{d}+q) (|\jbfm|_{\infty} + |\jbfm'|_{\infty})}
		\deltax(\lambdabfm,\lambdabfm')^{-(1+2q+2\tilde{d})}.
	\end{equation*}
\end{corollary}

These estimates will be useful when the supports of the wavelets involved are,
at least in one coordinate direction, well-separated and located sufficiently
far away from each other.
If this is not the case, we need estimates of the following kind.

\begin{lemma}
	\label{lm:near_field_previous}
	Assume that \(0 < \sigmax(\lambdabfm,\lambdabfm') \leq
	2^{-\mx(\lambdabfm,\lambdabfm')}\).
	Then, there holds
	\begin{align*}
		\big|\big\langle \psi_{\lambdabfm'}, \Lcal \psi_{\lambdabfm}
		\big\rangle\big| 
		&\lesssim 2^{-\frac{1}{2}|\jbfm -\jbfm'|_1}
		2^{-\tilde{d}(\Mx(\jbfm,\jbfm') +\My(\jbfm,\jbfm'))} 
		2^{-q(|\jbfm|_{\infty} + |\jbfm'|_{\infty})}
		\sigmax^{-(2q + 2\tilde{d})}.
	\end{align*}
	Conversely, if \(0 < \sigmay(\lambdabfm,\lambdabfm') \leq
	2^{-\my(\lambdabfm,\lambdabfm')}\), there holds
	\begin{align*}
		\big|\big\langle \psi_{\lambdabfm'}, \Lcal \psi_{\lambdabfm}
		\big\rangle\big| 
		&\lesssim 2^{-\frac{1}{2}|\jbfm -\jbfm'|_1}
		2^{-\tilde{d}(\Mx(\jbfm,\jbfm') +\My(\jbfm,\jbfm'))} 
		2^{-q(|\jbfm|_{\infty} + |\jbfm'|_{\infty})}
		\sigmay^{-(2q + 2\tilde{d})}.
	\end{align*}
\end{lemma}
\begin{proof}
	Let us assume without loss of generality that \(\jx' \geq \jx\).
	We exploit the vanishing moments of \(\psi_{\lambdax'}\) to obtain that
	\begin{align}
		\label{eq:estimate_with_mod_kernel}
		\big|\big\langle \psi_{\lambdabfm'}, \Lcal \psi_{\lambdabfm}
		\big\rangle\big| 
		\lesssim 2^{-(\tilde{d} + \frac{1}{2})\jx'} 2^{-q(|\jbfm|_{\infty} +
		|\jbfm'|_{\infty})}
		\esssup_{x' \in \Omega_{\lambdax'}} 
		\bigg|\partial_{x'}^{\tilde{d}}  \int_{\Omega_{\lambdax}} \kappaxt(x, x') 
		\psi_{\lambdax}(x) \dint x \bigg|,
	\end{align}
	where
	\begin{equation*}
		\kappaxt(x,x') \isdef \int_{\Omega_{\lambday}} \int_{\Omega_{\lambday'}}
		\psi_{\lambday}(y) \psi_{\lambday'}(y') \kappa(x, y, x', y')
		\dint y' \dint y.
	\end{equation*}

	In what follows, let us assume without loss of generality 
	that \(\jy' \geq \jy\).
	Then, for \(x \neq x'\), we have
	\begin{align*}
		\left|\partial_{x}^{\alpha} \partial_{x'}^{\alpha'} 
		\kappaxt(x,x')\right|
		&\lesssim 2^{-(\tilde{d} + \frac{1}{2})\jy'} \esssup_{y' \in \Omega_{\lambday'}}
		\bigg| \partial_{y'}^{\tilde{d}} \int_{\Omega_{\lambday}}
		\partial_{x}^{\alpha} \partial_{x'}^{\alpha'} \kappa(x, y, x', y')
		\psi_{\lambday}(y) \dint y \bigg| \\
		&\lesssim 2^{-\tilde{d}\jy'} 2^{-\frac{1}{2} (\jy' -\jy)} \esssup_{y' \in \Omega_{\lambday'}}
		\int_{\Omega_{\lambday}} \|\xvec -\xvec'\|^{-(2 + 2q + \tilde{d} + 
		\alpha + \alpha')} \dint y.
	\end{align*}
	Now, for any \(y' \in \Omega_{\lambday'}\), we infer that there holds
	\begin{align*}
		&\int_{\Omega_{\lambday}} \|\xvec -\xvec'\|^{-(2 + 2q +
		\tilde{d} + \alpha + \alpha')} \dint y\\
		&\qquad\qquad \lesssim \int_{0}^{|x - x'|} |x - x'|^{-(2 + 2q +
		\tilde{d} + \alpha + \alpha')}
		\dint t
		+ \int_{|x - x'|}^{\infty} t^{-(2 + 2q + \tilde{d} + \alpha +
		\alpha')} \dint t\\
		&\qquad\qquad \lesssim |x - x'|^{-(1 + 2q + \tilde{d} + \alpha + \alpha')}.
	\end{align*}
	Combined, we arrive at
	\begin{align}
		\label{eq:calderon_zygmund_decay}
		\left|\partial_{x}^{\alpha} \partial_{x'}^{\alpha'} 
		\kappaxt(x,x')\right|
		\lesssim 2^{-\tilde{d}\My(\jbfm,\jbfm')} 
		2^{-\frac{1}{2}|\jy -\jy'|} |x-x'|^{-(1 + 2q + \tilde{d} 
		+ \alpha + \alpha')}.
	\end{align}

	In the next step, we argue along the lines of \cite{DHS06,Rei08}, i.e.,
	we decompose \(\psi_{\lambdax} = \fxt + \fxb\), where \(\fxt\) is a smooth
	extension of \(\psi_{\lambdax}\) which agrees with it on the smooth part
	where \(\psi_{\lambdax'}\) is located.
	Note that, since \(\sigmax(\lambdabfm,\lambdabfm') > 0\),
	there holds \(\sigmax(\lambdabfm,\lambdabfm') \gtrsim 2^{-\jx'}\).
	Hence, \(\fxt\) coincides with \(\psi_{\lambdax}\) on a neighbourhood of
	\(\psi_{\lambdax'}\) such that \(\Dist(\Supp \fxb, \linebreak 
	\Omega_{\lambdax'}) \geq \sigmax(\lambdabfm,\lambdabfm')\).
	Additionally, there holds \(\|\fxb\|_{L^{\infty}(\square)} \lesssim
	2^{\frac{1}{2}\jx}\), and \(\|\fxt\|_{H^{s}(\square)} \lesssim 2^{s \jx}\).
	Therefore, for \(x' \in \Omega_{\lambdax'}\), there holds
	\begin{align}
		\notag
		&\left|\partial_{x'}^{\tilde{d}} \int_{\Omega_{\lambdax}}
		\kappaxt(x, x') \fxb(x) \dint x \right| \\
		\notag
		&\qquad \lesssim 2^{-\tilde{d}\My(\jbfm,\jbfm')} 2^{-\frac{1}{2}|\jy -\jy'|}
		\int_{\Supp \fxb} |x - x'|^{-(1 + 2q + 2\tilde{d})}
		\left|\fxb(x)\right| \dint x \\
		\notag
		&\qquad \lesssim 2^{-\tilde{d}\My(\jbfm,\jbfm')} 2^{-\frac{1}{2}|\jy -\jy'|}
		2^{\frac{1}{2} \jx} \int_{\sigmax(\lambdabfm,\lambdabfm')}^{\infty}
		t^{-(1 + 2q + 2\tilde{d})} \dint t \\
		\label{eq:estimate_fxb}
		&\qquad \lesssim 2^{-\tilde{d}\My(\jbfm,\jbfm')} 2^{-\frac{1}{2}|\jy -\jy'|}
		2^{\frac{1}{2} \jx} \sigmax(\lambdabfm,\lambdabfm')^{-(2q + 2\tilde{d})}.
	\end{align}

	We consider next the operator \(\Lcal_{\mathrm{x}}\) 
	with kernel \(\kappaxt\) 
	and an extension \(\Lcal_{\mathrm{x}}^{\sharp}\) thereof,
	defined by
	\begin{equation*}
		\groupp[\big]{\Lcal_{\mathrm{x}}^{\sharp} u}(x)
		\isdef \int_{\Rbbb} \chi(x) \chi(x') \kappaxt(x, x') u(x')
		\dint x,
	\end{equation*}
	where \(\chi \in C^{\infty}_{0}(\Rbbb)\) is a suitable cutoff function
	with \(\chi\big|_{[0, 1]} = 1\) and \(\chi\big|_{\Rbbb \setminus
	[-1, 2]} = 0\).
	Moreover, following the argumentation of 
	\cite[Lemma 2.1.4 \& Theorem 2.1.7]{Rei08},
	the estimate \eqref{eq:calderon_zygmund_decay} implies that
	\(\Lcal_{\mathrm{x}}^{\sharp}\) is continuous 
	of order \(2q + \tilde{d}\) 
	with norm 
	\(2^{-\tilde{d}\My(\jbfm,\jbfm')} \linebreak
	2^{-\nicefrac{1}{2} |\jy - \jy'|}\),
	which implies for \(\fxt\) that
	\begin{align*}
		\norms[\big]{\Lcal_{\mathrm{x}}^{\sharp} \fxt}_{W^{\tilde{d},
		\infty}(\Omega_{\lambdax'})} 
		&\lesssim 2^{-\tilde{d}\My(\jbfm,\jbfm')} 
		2^{-\frac{1}{2} |\jy - \jy'|}
		\norms[\big]{\fxt}_{W^{2q + 2\tilde{d}, \infty}(\Rbbb)} \\
		&\lesssim 2^{-\tilde{d}\My(\jbfm,\jbfm')} 
		2^{-\frac{1}{2} |\jy - \jy'|}
		2^{(\frac{1}{2} + 2q + 2\tilde{d})\jx}.
	\end{align*}
	Additionally using \(\sigmax(\lambdabfm,\lambdabfm') \lesssim 2^{-\jx}\),
	a combination of these estimates yields
	\begin{align}
		\notag
		\left|\partial_{x'}^{\tilde{d}} \int_{\Omega_{\lambdax}}
		\kappaxt(x, x') \fxt(x) \dint x \right| 
		&\lesssim 2^{-\tilde{d}\My(\jbfm,\jbfm')} 2^{-\frac{1}{2}|\jy -\jy'|}
		\big| \fxt \big|_{W^{2q + 2\tilde{d},\infty}(\Omega_{\lambdax'})}\\
		&\lesssim 2^{-\tilde{d}\My(\jbfm,\jbfm')} 2^{-\frac{1}{2}|\jy -\jy'|}
		\notag
		2^{(\frac{1}{2} + 2q + 2\tilde{d})\jx}\\
		\label{eq:estimate_fxt}
		&\lesssim 2^{-\tilde{d}\My(\jbfm,\jbfm')} 2^{-\frac{1}{2}|\jy -\jy'|}
		2^{\frac{1}{2}\jx} \sigmax(\lambdabfm,\lambdabfm')^{-(2q + 2\tilde{d})}.
	\end{align}

	By inserting \eqref{eq:estimate_fxb} and
	\eqref{eq:estimate_fxt} into
	\eqref{eq:estimate_with_mod_kernel}, we obtain the that
	\begin{align*}
		\big|\big\langle \psi_{\lambdabfm'}, \Lcal \psi_{\lambdabfm}
		\big\rangle\big| 
		&\lesssim 2^{-\frac{1}{2}|\jbfm -\jbfm'|_1}
		2^{-\tilde{d}(\Mx(\jbfm,\jbfm') +\My(\jbfm,\jbfm'))} 
		2^{-q(|\jbfm|_{\infty} + |\jbfm'|_{\infty})}
		\sigmax^{-(2q + 2\tilde{d})},
	\end{align*}
	which is what we wanted to show.
	The estimate in the second coordinate direction is validated in complete
	analogy.
\end{proof}

\begin{corollary}
	\label{cor:near_field_previous_only_1_vanishing}
	We have the estimates
	\begin{align*}
		\big|\big\langle \psi_{\lambdabfm'}, \Lcal \psi_{\lambdabfm}
		\big\rangle\big| 
		&\lesssim 2^{-\frac{1}{2}|\jbfm -\jbfm'|_1}
		2^{-\tilde{d}\maxm(\jbfm,\jbfm')}
		2^{-q(|\jbfm|_{\infty} + |\jbfm'|_{\infty})} 
		\sigmax(\lambdabfm,\lambdabfm')^{-(2q + \tilde{d})}, \\
		\big|\big\langle \psi_{\lambdabfm'}, \Lcal \psi_{\lambdabfm}
		\big\rangle\big| 
		&\lesssim 2^{-\frac{1}{2}|\jbfm -\jbfm'|_1}
		2^{-\tilde{d}\maxm(\jbfm,\jbfm')}
		2^{-q(|\jbfm|_{\infty} + |\jbfm'|_{\infty})} 
		\sigmay(\lambdabfm,\lambdabfm')^{-(2q + \tilde{d})},
	\end{align*}
	provided that \(0 < \sigmax(\lambdabfm,\lambdabfm') \lesssim
	2^{-\mx(\lambdabfm,\lambdabfm')}\) or
	\(0 < \sigmay(\lambdabfm,\lambdabfm') \lesssim
	2^{-\my(\lambdabfm,\lambdabfm')}\),
	respectively.
\end{corollary}
\begin{proof}
	We use the same arguments as in the previous proof, but do not 
	exploit vanishing moments in one of the coordinate directions.
\end{proof}

\begin{lemma}
	\label{lm:mixed_field_both_far}
	Assume that \(0 < \sigmax(\lambdabfm,\lambdabfm') \leq
	2^{-\mx(\lambdabfm,\lambdabfm')}\) and, additionally, that
	\(\deltay(\lambdabfm,\lambdabfm') > 0\).
	Then, we have
	\begin{equation}
		\label{eq:cor_mixed_field_both_far}
		\begin{aligned}
			\big|\big\langle \psi_{\lambdabfm'}, \Lcal \psi_{\lambdabfm}
			\big\rangle\big|
			&\lesssim 2^{-\frac{1}{2}(\jy + \jy' + |\jx -\jx'|)} 
			2^{-\tilde{d}(\Mx(\jbfm,\jbfm') + \My(\jbfm,\jbfm'))}\\
			&\qquad \cdot 2^{-q(|\jbfm|_\infty + |\jbfm'|_{\infty})} 
			\deltay(\lambdabfm,\lambdabfm')^{-(1 + q + \tilde{d})} 
			\sigmax(\lambdabfm,\lambdabfm')^{-(q + \tilde{d})}.
		\end{aligned}
	\end{equation}
\end{lemma}
\begin{proof}
	We follow the proof of Lemma \ref{lm:near_field_previous},
	except for the estimate \eqref{eq:calderon_zygmund_decay} on \(\kappaxt\).
	Instead, we note that by Young's inequality,
	\begin{align*}
		\left|\partial_{x}^{\alpha} \partial_{x'}^{\alpha'} 
		\kappaxt(x,x')\right|
		&\lesssim 2^{-\tilde{d}\My(\jbfm,\jbfm')} 2^{-\frac{1}{2}(\jy + \jy')} 
		\esssup_{(y,y') \in \Omega_{\lambday} \times \Omega_{\lambday'}} 
		\|\xvec -\xvec'\|^{-(2 + 2q + \tilde{d} + \alpha + \alpha')}\\
		&\lesssim 2^{-\tilde{d}\My(\jbfm,\jbfm')} 2^{-\frac{1}{2}(\jy + \jy')}
		\deltay(\lambdabfm,\lambdabfm')^{-(1+q+\tilde{d})} 
		|x-x'|^{-(1 + q + \alpha + \alpha')}.
	\end{align*}
	Also here, we can proceed along the lines of \cite{DHS06,Rei08}, 
	with the adaptation that the one-dimensional integral operator 
	with kernel \(\kappaxt\) is now of order \(q\) and the constant
	is given by \(2^{-\tilde{d} \My(\jbfm,\jbfm')} 2^{-\frac{1}{2}(\jy + \jy')}
	\deltay(\lambdabfm,\lambdabfm')^{-(1 + q + \tilde{d})}\).
	This results in \eqref{eq:cor_mixed_field_both_far}.
\end{proof}

\begin{corollary}
	\label{cor:mixed_field_both_far_second} 
	Assume that the requirements of Lemma \ref{lm:mixed_field_both_far}
	hold and, additionally, that \(\mx(\jbfm,\jbfm') \leq \my(\jbfm,\jbfm')\).
	Then, there holds
	\begin{align}
		\label{eq:cor_mixed_field_both_far_second}
		\big|\big\langle \psi_{\lambdabfm'}, \Lcal \psi_{\lambdabfm}
		\big\rangle\big|
		&\lesssim 2^{-\frac{1}{2}(\jy + \jy' + |\jx -\jx'|)} 
		2^{-(\tilde{d}+q)(|\jbfm|_{\infty} + |\jbfm'|_{\infty})}
		\deltay(\lambdabfm,\lambdabfm')^{-(1 + q + \tilde{d})} 
		\sigmax(\lambdabfm,\lambdabfm')^{-(q + \tilde{d})}.
	\end{align}
\end{corollary}
\begin{proof}
	Without loss of generality, let us assume that \(\jx \leq \jx'\).
	Then, since \(\mx(\jbfm,\jbfm') \leq \my(\jbfm,\jbfm')\), 
	we immediately see that \(|\jbfm|_{\infty} = \jy\).
	Hence, if \(|\jbfm'|_{\infty} = \jx'\),
	\eqref{eq:cor_mixed_field_both_far_second} immediately follows from
	\eqref{eq:cor_mixed_field_both_far}.

	On the other hand, if \(|\jbfm|_{\infty} = \jy'\),
	we do not use any vanishing moments in the direction of \(\mathrm{x}\),
	but only derive the estimate
	\begin{align*}
		\big|\big\langle \psi_{\lambdabfm'}, \Lcal \psi_{\lambdabfm}
		\big\rangle\big| 
		\lesssim 2^{-\frac{1}{2}\jx'} 2^{-q(|\jbfm|_{\infty} +
		|\jbfm'|_{\infty})}
		\esssup_{x' \in \Omega_{\lambdax'}} 
		\left| \int_{\Omega_{\lambdax}} \kappaxt(x, x') 
		\psi_{\lambdax}(x) \dint x \right|,
	\end{align*}
	and then conclude that the kernel \(\kappaxt\) satisfies
	\begin{align*}
		\left|\partial_{x}^{\alpha} \partial_{x'}^{\alpha'} 
		\kappaxt(x,x')\right|
		&\lesssim 2^{-(\tilde{d}+\frac{1}{2})(\jy + \jy')} 
		\esssup_{(y,y') \in \Omega_{\lambday} \times \Omega_{\lambday'}} 
		\|\xvec -\xvec'\|^{-(2 + 2q + 2\tilde{d} + \alpha + \alpha')}\\
		&\lesssim 2^{-(\tilde{d}+\frac{1}{2})(\jy + \jy')} 
		\deltay(\lambdabfm,\lambdabfm')^{-(1+q+\tilde{d})} 
		|x-x'|^{-(1 + q + \tilde{d} + \alpha + \alpha')}.
	\end{align*}
	Also here, we can proceed along the lines of \cite{DHS06,Rei08}, 
	which results in \eqref{eq:cor_mixed_field_both_far_second}.
\end{proof}

\begin{remark}
	We remark that Lemma \ref{lm:mixed_field_both_far}
	and Corollary \ref{cor:mixed_field_both_far_second} can also 
	be adapted to the opposite coordinate directions if \(\mathrm{x}\) and
	\(\mathrm{y}\) are interchanged.
\end{remark}

\section{Compression Scheme}
\label{sec:compression_general}

In this section, we are going to develop a compression 
scheme for the integral operator \(\Lcal\),
on which we will assume the following shift property
on the unit square \(\square\) and the boundary \(\Gamma\) of a
piecewise smooth Lipschitz manifold \(\Omega\).
\begin{assumption}
	\label{ass:continuity_of_L}
	There exists a \(\sigma > 0\) such that
	\(\Lcal\) and its dual operator \(\Lcal'\) are continuous operators
	\begin{align}
		\label{eq:anisotropic_shift_property}
		\Lcal\colon \Hfrak^{q,\alphabfm}(D) \to 
		\Hfrak^{-q, \alphabfm}(D),
		\qquad
		\Lcal'\colon \Hfrak^{q,\alphabfm}(D) \to 
		\Hfrak^{-q, \alphabfm}(D),
	\end{align}
	for all \(\alphabfm \geq \zerobfm\) such that \(\norms{\alphabfm}_1 < \sigma\),
	where \(D \in \{\square,\, \Gamma\}\).
\end{assumption}

\begin{remark}
	Classical integral operators are known to satisfy the shift property 
	\(\Lcal\colon H^{q + t}(\Gamma) \to H^{-q + t}(\Gamma)\) for 
	\(|t| < \sigma\) depending on the smoothness of the underlying manifold. 
	Especially, in view of \cite{Cos88,SS11}, 
	we have \(\sigma \geq \nicefrac{1}{2}\).
	As moreover the commutator \([\Lcal,\partial_{\xvec}^{\alphabfm}]
		\isdef \Lcal \partial_{\xvec}^{\alphabfm} 
	- \partial_{\xvec}^{\alphabfm}\Lcal\)
	inherits its properties from \(\Lcal\) in accordance with \cite{SW99},
	the isotropic shift implies the anisotropic shift property
	\eqref{eq:anisotropic_shift_property} as long as
	\(|\alphabfm|_1 < \sigma\).
\end{remark}

Under Assumption \ref{ass:continuity_of_L}, 
we will show that, for each \(r \in \Nbbb\),
the operator matrix \(\Lbfm\) can be compressed into a matrix \(\Lbfm_r\) 
such that there holds \(\|\Lbfm -\Lbfm_r\|_2 \lesssim 2^{-sr}\) for any
\begin{equation}
	\label{eq:definition_of_s_star}
	s < s^\star \isdef \min\left\{ \sigma,\ 
		\alpha \bar{s},\ 
		\frac{\tilde{d}}{2} + q - 
		\max\left\{0,\ q\alpha \frac{\bar{s}}{\nu(\gamma, q)}\right\}
	\right\}.
\end{equation}
Herein, \(\alpha > 1\) is an arbitrary but fixed parameter.
Moreover, \(\bar{s}\) denotes the sparse-grid approximation rate
\cite{GK00} and \(\nu(\gamma, q)\) denotes a parameter depending on
the regularity \(\gamma\) and the operator order \(q\). 
They are given by
\begin{align}
	\label{eq:sbar_nu}
	\bar{s} \isdef
	\begin{cases}
		d-q, &q \geq 0,\\
		d - \nicefrac{q}{2}, &q < 0,
	\end{cases}
	&&
	\nu(\gamma, q) \isdef
	\begin{cases}
		\gamma -q, &q \geq 0,\\
		\gamma, &q < 0.
	\end{cases}
\end{align}
For simplicity, we shall shortly write \(\nu \isdef \nu(\gamma, q)\).

Meanwhile, each row and column of \(\Lbfm_r\) contains no more than 
\(\Ocal(r^2 2^r)\) nontrivial entries. This implies the \(s^\star\)-compressibility 
of the bi-infinite matrix \(\Lbfm\), compare \cite{CDD02}. 
Indeed, if the compressed matrices \(\Lbfm_r\) satisfy the above conditions, 
it is easy to conclude that \(\tilde{\Lbfm}_r \isdef \Lbfm_{\lceil r + 
\log_2(\alpha_r)\rceil}\) satisfies Definition~\ref{def:s_star_compressibility}
with for example \(\alpha_r \isdef r^{-(3+\epsilon)}\) for some small 
\(\epsilon > 0\). 
This was in a similar way also remarked in 
\cite{Ste04}.
Note that we especially have \(s^{\star} > \bar{s}\)
if \(\sigma\) and \(\tilde{d}\) are sufficiently large.
Hence, the adaptive algorithm will produce approximations converging 
with the approximation rate provided by sparse grid approximation
for arbitrarily smooth functions,
cf.\ \cite{CDD01,CDD02,SS08}.

For the sake of convenience, we are going to treat all entries as if 
they corresponded to wavelet functions. 
This especially means that we are going to assume that we 
have a cancellation property for all functions, 
even if the underlying functions are 
scaling functions on level \(j_0\). 
Indeed, if for example \(\psi_{\lambdabfm} = 2^{-q|\lambdabfm|_{\infty}} 
\phi_{\lambdax} \otimes \psi_{\lambday}\), we may always use the estimate
\begin{align*}
	\big|\big\langle \Lcal \psi_{\lambdabfm'}, \psi_{\lambdabfm}
	\big\rangle\big| 
	&\lesssim 2^{-\frac{1}{2} j_0 -(\tilde{d} + \frac{1}{2})\jy} 
	2^{-q|\lambdabfm|_{\infty}}
	\left|  \Lcal \psi_{\lambdabfm'}
	\right|_{W^{\tilde{d},\infty}(\Omega_{\lambdabfm})}\\
	&\lesssim 2^{-(\tilde{d} + \frac{1}{2})|\jbfm|_{1}} 
	2^{-q|\lambdabfm|_{\infty}}
	\left|  \Lcal \psi_{\lambdabfm'}
	\right|_{W^{2\tilde{d},\infty}(\Omega_{\lambdabfm})},
\end{align*}
since \(2^{-\tilde{d}j_0}\) is a constant and the lower order derivative of  
an asymptotically smooth kernel can be estimated by a higher order derivative.
For all other required estimates, we can use a similar argument.

\subsection{Diagonal Cutoff}

First, we are going to drop the entries far from the diagonal
of the matrix. 
A possible approach is keeping only entries \(\ell_{\lambdabfm,\lambdabfm'}\)
for which \(|\lambdabfm -\lambdabfm'|_{\infty}\) is sufficiently small.

\begin{theorem}
	For \(\alpha > 1\), let us denote \(\Lbfm_r^{(0)}\) 
	as the matrix with the entries
	\begin{equation}
		\big[\ell_r ^{(0)} \big]_{\lambdabfm, \lambdabfm'} \isdef
		\begin{cases}
			\langle \Lcal \psi_{\lambdabfm'}, 
			\psi_{\lambdabfm} \rangle_{\square}, & |\lambdabfm -
			\lambdabfm'|_{\infty} \leq \alpha r
			\frac{\bar{s}}{\nu}, \\
			0, &\text{else}.
		\end{cases}
		\label{eq:compression_diagonal_cutoff}
	\end{equation}
	Then, we have the error estimate \(\|\Lbfm - \Lbfm_r^{(0)}\| \lesssim 2^{-sr}\)
	for any \(s < \min\{\sigma,\, \alpha \bar{s} \}\).
	\label{thm:diagonal_cutoff}
\end{theorem}

\begin{proof}
	Let \(s < \min\{\sigma,\, \alpha \bar{s} \} \), choose
	\(\frac{\nu}{\alpha \bar{s}} s < t < \min\{\sigma,\,
	\nu\}\),
	and consider \(\ubfm, \vbfm \in \ell^2(\nablabfm)\).
	We will show that
	\begin{equation*}
		\big|\langle\ubfm, \, (\Lbfm - \Lbfm_r^{(0)})\vbfm\rangle_{\ell^2(\nablabfm)} \big|
		\lesssim 2^{-sr} \|\ubfm\|_{\ell^2(\nablabfm)} \|\vbfm\|_{\ell^2(\nablabfm)}.
	\end{equation*}
	To this end, let us denote
	\begin{equation*}
		u_{\jbfm} = \sum_{\lambdabfm\in\nablabfm_{\jbfm}} u_{\lambdabfm} \psi_{\lambdabfm},
		\quad v_{\jbfm'} = \sum_{\lambdabfm'\in\nablabfm_{\jbfm'}} 
		v_{\lambdabfm'} \psi_{\lambdabfm'}.
	\end{equation*}

	Depending on \(\lambdabfm\) and \(\lambdabfm'\), let us define
	\begin{equation*}
		\tbfm \isdef (\tx,\ty) \isdef
		\begin{cases}
			(t \Sign(\jx' -\jx),\ 0), &|\jbfm -\jbfm'|_{\infty} = |\jx -\jx'|,\\
			(0,\ t \Sign(\jy' -\jy)), &|\jbfm -\jbfm'|_{\infty} = |\jy -\jy'|.
		\end{cases}
	\end{equation*}
	If \(\tbfm \leq \zerobfm\),
	then \(\Lcal \colon \Hfrak^{q,-\tbfm}(\square) \to \Hfrak^{-q, -\tbfm}(\square)\) 
	is continuous
	due to Assumption \ref{ass:continuity_of_L}
	since \(|\tbfm|_{1} < \sigma\), 
	so
	\begin{align*}
		\big| \langle \Lcal v_{\jbfm'}, \, u_{\jbfm} \rangle_{\square} \big|
		&\leq \|\Lcal v_{\jbfm'}\|_{\Hfrak^{-q, -\tbfm}(\square)}
		\|u_{\jbfm}\|_{\Hfrak^{q, \tbfm}(\square)}
		\lesssim \|v_{\jbfm'}\|_{\Hfrak^{q, -\tbfm}(\square)}
		\|u_{\jbfm}\|_{\Hfrak^{q, \tbfm}(\square)} \\
		&\lesssim 2^{\tx (\jx - \jx') + \ty (\jy - \jy')} 
		\|v_{\jbfm'}\|_{H^q(\square)} \|u_{\jbfm}\|_{H^q(\square)}.
	\end{align*}
	Herein, we have used the norm equivalence \eqref{eq:norm_equivalence_GK}
	since \(t < \nu(\gamma, q)\) implies that \(t, q+t <
	\gamma\).
	If \(\tbfm \geq \zerobfm\), on the other hand, we may apply the above
	argument to \(\norms{\groupa{v_{\jbfm'}, 
	\Lcal' u_{\jbfm}}_{\square}}\).
	In any case, we find that
	\begin{equation*}
		\big| \langle \Lcal v_{\jbfm'}, \, u_{\jbfm} \rangle_{\square} \big|
		\lesssim 2^{-t |\jbfm - \jbfm'|_{\infty}}
		\|\vbfm\|_{\ell^2(\nablabfm_{\jbfm'})} \|\ubfm\|_{\ell^2(\nablabfm_{\jbfm})}.
	\end{equation*}
	Hence, by summing over \(\jbfm\) and \(\jbfm'\), we obtain that
	\begin{align*}
		\big| \langle \ubfm, \, (\Lbfm - \Lbfm_r^{(0)}) 
		\vbfm \rangle_{\ell^2(\nablabfm)} \big|
		&\leq \sum_{|\jbfm - \jbfm'|_{_\infty} > \alpha r
		\frac{\bar{s}}{\nu}} 
		\big| \langle \Lcal v_{\jbfm'}, \, u_{\jbfm} \rangle_{\square} \big| \\
		&\lesssim \sum_{|\jbfm - \jbfm'|_{_\infty} > \alpha r
		\frac{\bar{s}}{\nu}} 
		2^{-t|\jbfm - \jbfm'|_{\infty}}
		\|\vbfm\|_{\ell^2(\nablabfm_{\jbfm'})} \|\ubfm\|_{\ell^2(\nablabfm_{\jbfm})} \\
		&= \bar{\vbfm}^{\intercal} \bar{\Abfm}_r \bar{\ubfm},
	\end{align*}
	where we denote \(\bar{\ubfm} = [ \|\ubfm\|_{\ell^2(\nablabfm_{\jbfm})}]
	_{\jbfm \in \Jcal}\) 
	and likewise \(\bar{\vbfm}\) with 
	\(\Jcal \isdef \{ \jbfm : \jbfm \geq \jbfm_0\}\), and
	\begin{equation*}
		\bar{\Abfm}_r = \big[ 2^{-t |\jbfm - \jbfm'|_{\infty}} 
			\mathds{1}_{|\jbfm - \jbfm'|_{_\infty} > \alpha r
		\frac{\bar{s}}{\nu}} \big]_{\jbfm, \jbfm' \in\Jcal}.
	\end{equation*}

	To estimate the spectral norm of \(\bar{\Abfm}_r\), 
	we remark that for any fixed \(\jbfm' \geq \jbfm_0\),
	the \(\jbfm'\)-th column sum can be estimated by
	\begin{align*}
		\sum_{|\jbfm - \jbfm'|_{_\infty} > \alpha r
		\frac{\bar{s}}{\nu}} 
		2^{-t|\jbfm - \jbfm'|_{\infty}} 
		&= \sum_{m > \alpha r \frac{\bar{s}}{\nu}}^{\infty} 
		2^{-tm} \big| \{ \jbfm \geq \jbfm_0 : |\jbfm - \jbfm'|_\infty = m \} \big|\\
		&\lesssim \sum_{m > \alpha r \frac{\bar{s}}{\nu}}^{\infty} 
		m 2^{-tm}\\
		&\lesssim 2^{-sr},
	\end{align*}
	as \(\alpha \frac{\bar{s}}{\nu} t > s\).
	With exactly the same argument, we can also estimate the row 
	sums of \(\bar{\Abfm}_r\), so by Schur's lemma, see for example 
	\cite{Sch98}, we have
	\begin{equation*}
		\|\bar{\Abfm}_r\|_2 
		\lesssim \|\bar{\Abfm}_r\|_1^{\frac{1}{2}} \|\bar{\Abfm}_r\|_{\infty}^{\frac{1}{2}}
		\lesssim 2^{-sr}.
	\end{equation*}
	Hence, there holds
	\begin{align*}
		\big| \langle \ubfm, \, (\Lbfm - \Lbfm_r^{(0)}) 
		\vbfm \rangle_{\ell^2(\nablabfm)} \big|
		&\lesssim \|\bar{\vbfm}\|_{\ell^2(\Jcal)} 
		\|\bar{\Abfm}_r\|_2 \|\bar{\ubfm}\|_{\ell^2(\Jcal)} 
		\lesssim 2^{-sr} \|\vbfm\|_{\ell^2(\nablabfm)} \|\ubfm\|_{\ell^2(\nablabfm)}.
	\end{align*}
\end{proof}

\begin{remark}
	As we will see, for the following error and complexity estimates,
	the condition \(|\jbfm -\jbfm'|_\infty > \alpha r \frac{\bar{s}}{\nu}\) 
	is not strictly necessary for leaving the desired number of
	remaining matrix entries.
	The only thing we require is that for any multiindex \(\lambdabfm \in
	\nablabfm_{\jbfm}\), 
	\begin{equation*}
		\left|\big\{\jbfm' : \exists \lambdabfm' \in \nablabfm_{\jbfm'}
			\text{ such that } \big[\ell_r^{(0)}\big]_{\lambdabfm,\lambdabfm'}
		\neq 0 \big\}\right| \lesssim r^{m}.
	\end{equation*}
	This condition is satisfied for m = 2 and any fixed \(\alpha > 0\).
\end{remark}

\subsection{Far-Field: First Compression}
\label{sec:first_compression}

In a second step, we want to drop the entries in \(\Lbfm\), 
for which the supports of the included wavelets
are located sufficiently far away from each other.
The corresponding cutoff parameter will 
be denoted by \(\Bcal_{\lambdabfm,\lambdabfm'}\).
Since this parameter does actually only
depend on the levels \(\jbfm\) and \(\jbfm'\)
of \(\lambdabfm\) and \(\lambdabfm'\), respectively,
we will also use the notation \(\Bcal_{\jbfm,\jbfm'}\)
if this is more feasible in the respective context.

\begin{theorem}
	For \(\nicefrac{1}{2}<\xi<1\), let us define
	\begin{equation}
		\label{eq:cutoff_parameter_first_general}
		\Bcal_{\lambdabfm,\lambdabfm'} \isdef
		\max\left\{2^{-\minm(\lambdabfm,\lambdabfm')}, \
			2^{-\frac{1}{2}(\mx(\lambdabfm,\lambdabfm')+\my(\lambdabfm,\lambdabfm'))}
		2^{\xi(\frac{r}{2} -|\lambdabfm-\lambdabfm'|_{\infty})}\right\}
	\end{equation}
	and \(\Lbfm_r^{(1)}\) as the infinite matrix with the entries
	\begin{equation}
		\label{eq:first_compression_scheme}
		[\ell_r^{(1)}]_{\lambdabfm,\lambdabfm'} \isdef
		\begin{cases}
			0, &\text{if } [\ell_r^{(0)}]_{\lambdabfm,\lambdabfm'}=0,\\
			0, &\text{if } \delta(\lambdabfm,\lambdabfm') \geq
			\Bcal_{\lambdabfm,\lambdabfm'},\\
			\langle\Lcal\psi_{\lambdabfm'},\psi_{\lambdabfm}\rangle_{\square}, &\text{else}.
		\end{cases}
	\end{equation}
	Then, there holds \(\|\Lbfm_r^{(0)}-\Lbfm_r^{(1)}\| \lesssim 
	r 2^{-\frac{\tilde{d} + q}{2} r}\).
	\label{thm:first_compression_error}
\end{theorem}
\begin{proof}
	Let us fix some \(\jbfm \geq \jbfm_0\) and \(\lambdabfm\in\nablabfm_{\jbfm}\).
	In what follows, we are going to use
	\eqref{eq:cancellation_property_two_vanmom} with \((j^{(1)}, j^{(2)}) =
	(\norms{\jbfm}_{\infty}, \norms{\jbfm'}_{\infty})\).
	Since \(\Bcal_{\lambdabfm,\lambdabfm'}\gtrsim 2^{-\minm(\lambdabfm,\lambdabfm')}\),
	we may follow the arguments of \cite{HvR24,Rei08} to obtain that
	the \(\lambdabfm\)-th row sum is bounded by
	\begin{align}
		\notag
		&\sum_{\lambdabfm'\in\nablabfm} 
		\big|[\ell_r^{(0)}-\ell_r^{(1)}]_{\lambdabfm,\lambdabfm'}\big|\\
		\notag
		&\qquad \lesssim 
		\sum_{|\jbfm - \jbfm'|_{_\infty} \leq \alpha r
		\frac{\bar{s}}{\nu}} 
		\sum_{\substack{\lambdabfm'\in\nabla_{\jbfm'}\\
		\delta(\lambdabfm,\lambdabfm') \geq \Bcal_{\jbfm,\jbfm'}}}
		2^{-\frac{1}{2}(|\lambdabfm|_1+|\lambdabfm'|_1)}
		2^{-(\tilde{d} + q)(|\lambdabfm|_\infty+|\lambdabfm'|_\infty)}
		\delta(\lambdabfm,\lambdabfm')^{-(2+2q+2\tilde{d})}\\
		\notag
		&\qquad \lesssim \sum_{|\jbfm - \jbfm'|_{_\infty} \leq \alpha r
		\frac{\bar{s}}{\nu}}
		2^{\frac{1}{2}(|\jbfm'|_1-|\jbfm|_1)}
		2^{-(\tilde{d} +  q)(|\jbfm|_\infty+|\jbfm'|_\infty)}
		\int_{\Bcal_{\jbfm,\jbfm'}}^{\infty} t^{-(1+2q+2\tilde{d})}\dint t\\
		\label{eq:maximum_to_bound_thm_42}
		&\qquad \lesssim \sum_{|\jbfm - \jbfm'|_{_\infty} \leq \alpha r
		\frac{\bar{s}}{\nu}}
		2^{\frac{1}{2}(|\jbfm'|_1-|\jbfm|_1)}
		2^{-(\tilde{d} +  q)(|\jbfm|_\infty+|\jbfm'|_\infty)}
		\Bcal_{\jbfm,\jbfm'}^{-(2q+2\tilde{d})}.
	\end{align}

	Next, we are going to distinguish between two cases.
	If \(|\jbfm -\jbfm'|_{\infty} \leq \nicefrac{r}{2}\), 
	then \(\Bcal_{\jbfm,\jbfm'}\) in
	\eqref{eq:maximum_to_bound_thm_42} can be bounded by
	\(2^{-\frac{1}{2}(\mx(\jbfm,\jbfm') + \my(\jbfm,\jbfm'))} 
	2^{\xi(\frac{r}{2} -|\jbfm-\jbfm'|_{\infty})}\), 
	so since 
	\begin{align*}
		\tilde{d} + q \geq 0, \qquad 
		|\jbfm|_{\infty} + |\jbfm'|_{\infty} - \mx(\jbfm,\jbfm') - \my(\jbfm,\jbfm') 
		\geq |\jbfm - \jbfm'|_{\infty},
	\end{align*}
	we have
	\begin{align*}
		\sum_{\substack{\lambdabfm'\in\nablabfm \\ |\lambdabfm -
		\lambdabfm'|_{\infty} \leq \frac{r}{2}}} 
		2^{\frac{1}{2}(|\lambdabfm|_1-|\lambdabfm'|_1)}
		\big|[\ell_r^{(0)}-\ell_r^{(1)}]_{\lambdabfm,\lambdabfm'}\big|
		&\lesssim 2^{-\xi (q + \tilde{d})r} \sum_{|\jbfm-\jbfm'|_{\infty}\leq
		\frac{r}{2}}
		2^{(2\xi-1)(q+\tilde{d})|\jbfm-\jbfm'|_{\infty}}\\
		&\lesssim r 2^{-\frac{\tilde{d} + q}{2}r}
	\end{align*}
	as \(\xi > \nicefrac{1}{2}\).

	On the other hand, if \(|\jbfm -\jbfm|_{\infty} \geq \nicefrac{r}{2}\), 
	then \(\Bcal_{\jbfm,\jbfm'}\) in \eqref{eq:maximum_to_bound_thm_42} is
	\(2^{-\minm(\jbfm,\jbfm')}\), so again by
	\begin{align*}
		\tilde{d} + q \geq 0, \qquad 
		|\jbfm|_{\infty} + |\jbfm'|_{\infty} - 2\minm(\jbfm,\jbfm')
		\geq |\jbfm - \jbfm'|_{\infty}
	\end{align*}
	there follows
	\begin{align*}
		\sum_{\substack{\lambdabfm'\in\nablabfm \\ |\lambdabfm -
		\lambdabfm'|_{\infty} \geq \frac{r}{2}}} 
		2^{\frac{1}{2}(|\lambdabfm|_1-|\lambdabfm'|_1)}
		\big|[\ell_r^{(0)}-\ell_r^{(1)}]_{\lambdabfm,\lambdabfm'}\big|
		\lesssim \sum_{|\jbfm-\jbfm'|_{\infty} \geq \frac{r}{2}}
		2^{-(q+\tilde{d})|\jbfm-\jbfm'|_{\infty}}
		\lesssim r 2^{- \frac{\tilde{d}+q}{2}r}.
	\end{align*}

	With exactly the same arguments, we get an analogous estimate for the column sums,
	i.e., for any fixed \(\lambdabfm'\in\nablabfm\), we have
	\begin{equation*}
		\sum_{\lambdabfm\in\nablabfm}
		2^{\frac{1}{2}(|\lambdabfm'|_1-|\lambdabfm|_1)}
		\big|[\ell_r^{(0)} - \ell_r^{(1)}]_{\lambdabfm,\lambdabfm'}\big|
		\lesssim r2^{- \frac{\tilde{d} + q}{2} r}.
	\end{equation*}
	Hence, we may use Schur's lemma to conclude that
	\(\|\Lbfm_r^{(0)}-\Lbfm_r^{(1)}\| \lesssim r2^{-\frac{\tilde{d} + q}{2} r}\),
	as long as \(\nicefrac{1}{2} < \xi < 1\).
\end{proof}

\subsection{Mixed Field: Mixed Compression}
\label{sec:mixed_compression}

Compared to isotropic wavelets, 
the requirement \(\delta(\lambdabfm,\lambdabfm') 
> 2^{-\minm(\lambdabfm,\lambdabfm')}\) is rather restrictive
but strictly required by the first compression from Section
\ref{sec:first_compression}.
On the other hand, there are still a lot of wavelet interactions
of the kind where \(\delta(\lambdabfm,\lambdabfm')>0\).
In the case where at least \(\deltax(\lambdabfm,\lambdabfm')\) 
or \(\deltay(\lambdabfm,\lambdabfm')\) is sufficiently large,
we can still somehow \enquote{first compress} the entries.
However, as we will see, the number of remaining entries is too big,
so we need to combine the first compression in one direction with
the second compression in the other direction. Note that this 
idea was already introduced in \cite{Rei08}.

\begin{theorem}
	\label{thm:compression_error_mix_y}
	For \(\nicefrac{1}{2} < \theta < 1\), let us define
	\begin{align}
		\label{eq:cutoff_mixed_field_y_far}
		\Dparam{\lambdabfm}{y} \isdef 2^{-\my(\lambdabfm, \lambdabfm')}
		\max\left\{1, \ 2^{\theta(\frac{r}{2} -
		|\jy-\jy'|)}\right\}, 
	\end{align}
	and \(\Lbfm_r^{(\frac{4}{3},{\rm y})}\) as the 
	infinite matrix with the entries
	\begin{equation}
		[\ell_r^{(\frac{4}{3},{\rm y})}]_{\lambdabfm,\lambdabfm'} \isdef
		\begin{cases}
			0, &\text{if }[\ell_r^{(1)}]_{\lambdabfm,\lambdabfm'}=0,\\
			0, &\text{if }
			\begin{cases}
				\deltay(\lambdabfm,\lambdabfm')\geq\Dparam{\lambdabfm}{y},\\
				\deltax(\lambdabfm,\lambdabfm')\leq 2^{-\mx(\lambdabfm,\lambdabfm')},\\
			\end{cases}\\
			\langle\Lcal\psi_{\lambdabfm'},\psi_{\lambdabfm}\rangle, &\text{else}.
		\end{cases}
		\label{eq:compression_scheme_mix_y}
	\end{equation}
	Then, there holds \(\|\Lbfm_r^{(1)}-\Lbfm_r^{(\frac{4}{3},{\rm y})} \|_2 
	\lesssim r 2^{-\frac{\tilde{d}+q}{2} r}\).
\end{theorem}
\begin{proof}
	We first remark that for any such compressed entry,
	we must have \(\my(\lambdabfm,\lambdabfm')\geq \mx(\lambdabfm,\lambdabfm')\).
	If this is not the case, then we immediately have
	\begin{equation*}
		\delta(\lambdabfm,\lambdabfm') \geq 
		\deltay(\lambdabfm,\lambdabfm') \geq \Dparam{\lambdabfm}{y}\geq 
		2^{-\my(\lambdabfm,\lambdabfm')}
		= 2^{-\minm(\lambdabfm,\lambdabfm')},
	\end{equation*}
	by which this entry is already set to zero in \(\Lbfm_r^{(1)}\).
	Moreover, since
	\(\Dparam{\lambdabfm}{y}\geq 2^{-\my(\lambdabfm,\lambdabfm')}\), 
	we can use Lemma \ref{lm:long_face_estimate} 
	to estimate the weighted, \(\lambdabfm\)-th row sum by
	\begin{align}
		\notag
		&\sum_{\lambdabfm'\in\nablabfm_{\jbfm'}} 
		2^{\frac{1}{2}(|\lambdabfm|_1-|\lambdabfm'|_1)} 
		\big|[\ell_r^{(1)}]_{\lambdabfm,\lambdabfm'}
		-[\ell_r^{(\frac{4}{3}, {\rm y})}]_{\lambdabfm,\lambdabfm'}\big|\\
		\notag
		&\qquad\lesssim 
		2^{\frac{1}{2}(\jx - \jx' - |\jx-\jx'|)}
		2^{-(\tilde{d}+q)(|\jbfm|_{\infty} + |\jbfm'|_{\infty})}
		\hspace{-1.5em}
		\sum_{\substack{\lambdabfm'\in\nablabfm_{\jbfm'}\\ 
		\deltay(\lambdabfm,\lambdabfm')\geq\Dparam{\lambdabfm}{y}}} 
		\hspace{-1.5em}
		2^{-\jy'}
		\deltay(\lambdabfm,\lambdabfm')^{-(1+2q+2\tilde{d})}.
	\end{align}
	In contrast to the proof of Theorem \ref{thm:first_compression_error},
	we can no longer estimate the sum by a bivariate integral and use polar
	coordinates.
	Instead, we note that the summands are independent of
	\(\lambdax'\),
	and  
	their number is bounded by \(\Ocal(2^{\jx' - \mx(\jbfm,\jbfm')})\)
	for fixed \(\lambday'\).
	For the summation in \(\mathrm{y}\), we can nontheless use an integral
	argument, so
	\begin{align}
		\notag
		&\sum_{\lambdabfm'\in\nablabfm_{\jbfm'}} 
		2^{\frac{1}{2}(|\lambdabfm|_1-|\lambdabfm'|_1)} 
		\big|[\ell_r^{(1)}]_{\lambdabfm,\lambdabfm'}
		-[\ell_r^{(\frac{4}{3}, {\rm y})}]_{\lambdabfm,\lambdabfm'}\big|\\
		\notag
		&\qquad \lesssim 2^{\frac{1}{2}(\jx - \jx' - |\jx-\jx'|)}
		2^{-(\tilde{d}+q)(|\jbfm|_{\infty} + |\jbfm'|_{\infty})} 
		2^{\jx'-\mx(\lambdabfm,\lambdabfm')}
		\int_{\Dparam{\lambdabfm}{y}}^{\infty} t^{-(1+2q+2\tilde{d})} \dint t\\
		\notag
		&\qquad\lesssim 
		2^{-(\tilde{d}+q)(|\jbfm|_{\infty} + |\jbfm'|_{\infty})} 		
		\big(\Dparam{\jbfm}{y}\big)^{-(2q+2\tilde{d})}\\
		\label{eq:maximum_in_thm_43}
		&\qquad\lesssim 
		2^{-(\tilde{d}+q)|\jy -\jy'|} 
		\max\left\{1, \ 2^{\theta(\frac{r}{2} -
		|\jy-\jy'|)}\right\}^{-(2q + 2\tilde{d})},
	\end{align}
	since \(\jx + \jx' = 2\mx(\jbfm,\jbfm') + \norms{\jx - \jx'}\),
	\(\tilde{d} + q \geq 0\), and 
	\(|\jbfm|_{\infty} +|\jbfm'|_{\infty} \geq \jy +\jy'\).

	We remark that the maximum in \eqref{eq:maximum_in_thm_43} is \(1\) if and
	only if
	\(|\jy - \jy'| \geq \nicefrac{r}{2}\), by which
	\begin{align*}
		\sum_{|\jy -\jy'| \geq \frac{r}{2}}
		\sum_{\lambdabfm' \in \Ical_{\lambdabfm,\jbfm'}}
		2^{\frac{1}{2}(|\jbfm|_1 - |\jbfm'|_1)}
		\left|\ell_{\lambdabfm,\lambdabfm'}\right|
		\lesssim 
		\sum_{|\jy -\jy'| \geq \frac{r}{2}}
		2^{-(\tilde{d} + q)|\jy -\jy'|}
		\lesssim r 2^{-\frac{\tilde{d} + q}{2}r}.
	\end{align*}
	Herein, \(\Ical_{\lambdabfm,\jbfm'}\) denotes the set of discarded entries
	of level \(\jbfm'\) in the \(\lambdabfm\)-th column.
	Additionally,
	in view of \(\theta > \frac{1}{2}\),
	\begin{align*}
		\sum_{|\jy -\jy'| \leq \frac{r}{2}}
		\sum_{\lambdabfm' \in \Ical_{\lambdabfm,\jbfm'}}
		2^{\frac{1}{2}(|\jbfm|_1 - |\jbfm'|_1)}
		\left|\ell_{\lambdabfm,\lambdabfm'}\right|
		&\lesssim 
		2^{-\theta(\tilde{d}+q) r} \sum_{|\jy -\jy'| \leq \frac{r}{2}}
		2^{|\jy-\jy'|(\tilde{d}+q)(2 \theta -1)}\\
		&\lesssim r 2^{-r\frac{\tilde{d} + q}{2}}.
	\end{align*}

	By deriving a similar estimate for the column sums and using Schur's lemma,
	we can conclude the statement of the theorem.
\end{proof}

Theorem \ref{thm:compression_error_mix_y} deals with wavelets,
for which the associated entries decay with the
distance of their supports in the second coordinate direction. 
If one of the involved wavelets is thin and long, 
we also need to compress some entries
along this long face, cf.\ Figure \ref{fig:support_picture}.
The requirement for a good compressibility is that
the support of the smaller wavelet with respect to the
first coordinate direction must be located sufficiently far away from the 
singular support of the long wavelet.

\begin{theorem}
	\label{thm:mixed_second_previous}
	Assume that \(|\jx -\jx'| \geq \nicefrac{r}{2}\). We define the cutoff parameter
	\begin{align}
		\label{eq:Eparam_previous}
		\Eparam{\jbfm}{x} \isdef 
		2^{-\mx(\jbfm,\jbfm')} 
		2^{\frac{r}{2} - |\jx -\jx'|},
	\end{align}
	and \(\Lbfm_r^{(\frac{5}{3},{\rm x})}\) as the 
	infinite matrix with the entries
	\begin{equation}
		\notag
		[\ell_r^{(\frac{5}{3},{\rm x})}]_{\lambdabfm,\lambdabfm'} \isdef
		\begin{cases}
			0, &\text{if }[\ell_r^{(1)}]_{\lambdabfm,\lambdabfm'}=0,\\
			0, &\text{if }
			\begin{cases}
				2^{-\mx(\lambdabfm,\lambdabfm')} \geq 
				\sigmax(\lambdabfm,\lambdabfm') \geq \Eparam{\lambdabfm}{x},\\
				2^{-\my(\lambdabfm,\lambdabfm')} \leq \deltay(\lambdabfm,\lambdabfm'),\\
			\end{cases}\\
			\langle\Lcal\psi_{\lambdabfm'},\psi_{\lambdabfm}\rangle, &\text{else}.
		\end{cases}
	\end{equation}
	Then, the error can be bounded by \(\|\Lbfm_r^{(\frac{4}{3},\mathrm{y})} - 
		\Lbfm_r^{(\frac{5}{3},\mathrm{x})}\|_2 \lesssim r^2 2^{-\frac{\tilde{d} +
	q}{2}r}\).
\end{theorem}
\begin{proof}
	For fixed \(\lambdabfm \in \nablabfm_{\jbfm}\) and \(\jbfm' \geq \jbfm_0\),
	let us denote the set of dropped entries in the \(\lambdabfm\)-th row or
	column by
	\(\Ical_{\lambdabfm,\jbfm'}\).
	We use Corollary \ref{cor:mixed_field_both_far_second} to bound
	\begin{equation*}
		\norms[\big]{\ell_{\lambdabfm,\lambdabfm'}}
		\lesssim 2^{-\frac{1}{2}(\jy + \jy' + \norms{\jx - \jx'})}
		2^{-(\tilde{d} + q)(\norms{\jbfm}_{\infty} + \norms{\jbfm'}_{\infty})}
		\deltay(\lambdabfm,\lambdabfm')^{-(1 + q + \tilde{d})}
		\groupp[\big]{\Eparam{\jbfm}{x}}^{-(q + \tilde{d})}.
	\end{equation*}
	Similar to the proof of Theorem \ref{thm:mixed_second_previous},
	we can conclude that
	\begin{align*}
		\sum_{\lambdabfm' \in \Ical_{\lambdabfm,\jbfm'}}
		2^{\frac{1}{2}(|\jbfm|_1 - |\jbfm'|_1)} \left|
		\ell_{\lambdabfm,\lambdabfm'} \right| 
		&\lesssim
		2^{\frac{1}{2}(|\jbfm|_1 - |\jbfm'|_1 - \jy -\jy' - |\jx
		-\jx'|)}
		2^{-(\tilde{d}+q)(|\jbfm|_{\infty} + |\jbfm'|_{\infty})}\\[-3ex]
		&\qquad\cdot
		2^{\jx'-\mx(\jbfm,\jbfm')} \big(\Eparam{\jbfm}{x}\big)^{-(q+\tilde{d})}
		2^{\jy'} \int_{2^{-\my(\jbfm,\jbfm')}}^{\infty} t^{-(1+q+\tilde{d})} \dint t\\
		&\lesssim
		2^{-(\tilde{d}+q)(|\jbfm|_{\infty} + |\jbfm'|_{\infty} -\my(\jbfm,\jbfm'))}
		\big(\Eparam{\jbfm}{x}\big)^{-(q+\tilde{d})} \\
		&\lesssim
		2^{-(\tilde{d}+q)(|\jbfm|_{\infty} + |\jbfm'|_{\infty} -\mx(\jbfm,\jbfm') -\my(\jbfm,\jbfm'))}
		\big( 2^{\frac{r}{2} - |\jx -\jx'|} \big)^{-(q+\tilde{d})}.
	\end{align*}
	Since \(|\jbfm|_{\infty} + |\jbfm'|_{\infty} - \my(\jbfm,\jbfm') -
	\mx(\jbfm,\jbfm') \geq |\jx -\jx'|\), we obtain that
	\begin{align*}
		\sum_{\lambdabfm' \in \Ical_{\lambdabfm,\jbfm'}}
		2^{\frac{1}{2}(|\jbfm|_1 - |\jbfm'|_1)} \left|
		\ell_{\lambdabfm,\lambdabfm'} \right| 
		\lesssim 2^{-\frac{\tilde{d} + q}{2}r}
		2^{-(\tilde{d} + q)|\jx-\jx|} 2^{(\tilde{d} + q)|\jx-\jx'|}
		= 2^{-\frac{\tilde{d} + q}{2}r}.
	\end{align*}
	Hence, summing over all \(\jbfm'\) which need to be considered, we
	obtain that the total error is bounded by \(\Ocal(r^2
	2^{-\frac{\tilde{d} + q}{2}r})\), and we can conclude using
	Schur's lemma.
\end{proof}

It is straightforward to exchange the coordinate directions and to define
\(\Lbfm^{(\frac{4}{3}, \mathrm{x})}\) and 
\(\Lbfm^{(\frac{5}{3}, \mathrm{y})}\), for which we have exactly the same error
estimates.
To summarise, we combine the previous theorems 
to obtain the main result of this subsection.

\begin{theorem}
	\label{thm:compression_error_mix_combined}
	Let us define the matrix \(\Lbfm_r^{(\textnormal{mix})}\) as
	\begin{align*}
		[\ell_r^{(\textnormal{mix})}]_{\lambdabfm,\lambdabfm'}\isdef
		\begin{cases}
			0, &\text{if }
			[\ell_r^{(\eta, {\rm z})}]_{\lambdabfm,\lambdabfm'}=0 
			\text{ for some }(\eta, {\rm z}) \in \{\frac{4}{3},
			\frac{5}{3}\} \times \{ {\rm x}, {\rm y} \},\\
			\langle\Lcal\psi_{\lambdabfm'}, \psi_{\lambdabfm}\rangle_\square,
			&\text{else}.
		\end{cases}
	\end{align*}
	Then, there holds \(\|\Lbfm_r^{(1)}-\Lbfm_r^{(\textnormal{mix})}\|_2 
	\lesssim r^2 2^{-\frac{\tilde{d} + q}{2} r}\).
\end{theorem}

\subsection{Near-Field: Second Compression}
\label{sec:second_compression}

In a final step, we need to develop a compression scheme,
in which we can discard entries associated with wavelet pairs
which are located very close to each other 
or even have overlapping supports. However, 
we must have \(\sigmax(\lambdabfm,\lambdabfm') >0\) or 
\(\sigmay(\lambdabfm,\lambdabfm') >0\) to drop such entries.

As we need to distinguish between three cases, 
we split the error estimate into lemmata.
As usual, we will only consider one coordinate direction, 
but by symmetry, analogous results hold for the other coordinate direction.

\begin{lemma}
	\label{lm:near_field_previous_first}
	Assume that 
	\((|\jbfm|_{\infty},\, |\jbfm'|_{\infty}) = (\jx,\,\jx')\).
	If all entries are dropped for which 
	\begin{align*}
		2^{-\mx(\lambdabfm,\lambdabfm')} &\geq	
		\sigmax(\lambdabfm,\lambdabfm') \geq \Fcal \isdef 
		2^{-\mx(\jbfm,\jbfm')} 2^{\frac{r}{2} - |\jx -\jx'|}, \\
		2^{-\my(\jbfm,\jbfm')} &\geq \deltay(\lambdabfm,\lambdabfm'),
	\end{align*}
	each appropriately weighted row or column sum of the error matrix
	is of order \(\Ocal(2^{-sr})\)
	for any \(s < \frac{\tilde{d}}{2} +q - \max\left\{0,\, q\alpha
	\nicefrac{\bar{s}}{\nu} \right\}\).
\end{lemma}
\begin{proof}
	Let us again fix \(\lambdabfm \in \nablabfm_{\jbfm}\) and \(\jbfm' \geq
	\jbfm_0\) and define the set of discarded entries of level \(\jbfm'\) 
	in the \(\lambdabfm\)-th row as \(\Ical_{\lambdabfm,\jbfm'}\).
	We note that \(\norms{\Ical_{\lambdabfm,\jbfm'}} \sim 2^{|\jbfm'|_1 -
	\mx(\jbfm,\jbfm') - \my(\jbfm,\jbfm')}\).
	Then, by using appropriate weights, we find that,
	using Corollary \ref{cor:near_field_previous_only_1_vanishing}, 	
	\begin{align*}
		\sum_{\lambdabfm' \in \Ical_{\lambdabfm,\jbfm'}}
		2^{\frac{1}{2}(|\jbfm|_1 - |\jbfm'|_1)}
		\left|\ell_{\lambdabfm,\lambdabfm'}\right|
		&\lesssim 2^{-\tilde{d}\maxm(\jbfm,\jbfm')} 2^{-q(\jx + \jx')} 
		\Fcal^{-(2q + \tilde{d})}\\[-0.3cm]
		&\lesssim 2^{-(2q + \tilde{d}) \frac{r}{2}}2^{-(q +\tilde{d})|\jx -
		\jx'|} 2^{(2q + \tilde{d})|\jx -\jx'|}\\
		&= 2^{-(q + \frac{\tilde{d}}{2})r} 2^{q|\jx-\jx'|}.
	\end{align*}
	Hence, if \(q \geq 0\), we have that
	\begin{align*}
		\sum_{|\jbfm -\jbfm'|_{\infty} \leq \alpha r
		\frac{\bar{s}}{\nu}} 
		\sum_{\lambdabfm' \in \Ical_{\lambdabfm,\jbfm'}}
		2^{\frac{1}{2}(|\jbfm|_1 - |\jbfm'|_1)}
		\left|\ell_{\lambdabfm,\lambdabfm'}\right|
		&\lesssim 2^{-(q + \frac{\tilde{d}}{2})r} \sum_{|\kbfm|_{\infty} \leq
		\alpha r \frac{\bar{s}}{\nu}} 2^{q\kx}\\
		&\lesssim r^2 2^{-(q + \frac{\tilde{d}}{2})r} 2^{q \alpha r
		\frac{\bar{s}}{\nu}}.
	\end{align*}
	Note that we can replace \(r^2\) by \(r\) if \(q\) is strictly positive.
	On the other hand, if \(q < 0\), then 
	\begin{align*}
		\sum_{|\jbfm -\jbfm'|_{\infty} \leq \alpha r
		\frac{\bar{s}}{\nu}} 
		\sum_{\lambdabfm' \in \Ical_{\lambdabfm,\jbfm'}}
		2^{\frac{1}{2}(|\jbfm|_1 - |\jbfm'|_1)}
		\left|\ell_{\lambdabfm,\lambdabfm'}\right|
		\lesssim r 2^{-(\frac{\tilde{d}}{2} + q)r}.
	\end{align*}
\end{proof}

\begin{lemma}
	\label{lm:near_field_previous_second}
	Assume (with a slight abuse of notation) that 
	\((|\jbfm|_{\infty},\, |\jbfm'|_{\infty}) 
	= \linebreak (\Mx(\jbfm,\jbfm'),\, \my(\jbfm,\jbfm'))\).
	If all entries which satisfy
	\begin{align*}
		2^{-\mx(\lambdabfm,\lambdabfm')} &\geq 
		\sigmax(\lambdabfm,\lambdabfm') \geq \Fcal \isdef 
		2^{-\frac{1}{2}(\mx(\jbfm,\jbfm') + \my(\jbfm,\jbfm'))}
		2^{\frac{1}{2} (r - |\jbfm -\jbfm'|_1)},\\
		2^{-\my(\lambdabfm,\lambdabfm')} &\geq \deltay(\lambdabfm,\lambdabfm'),
	\end{align*}
	are dropped, 
	each appropriately weighted row or column sum 
	of the error matrix is of order 
	\(\Ocal(2^{-sr})\), provided that \(s < \tilde{d} + q -
	\max\{0,\, q\alpha \nicefrac{\bar{s}}{\nu}\}\).
\end{lemma}
\begin{proof}
	We use Lemma \ref{lm:near_field_previous}
	and again \(\norms{\Ical_{\lambdabfm,\jbfm'}} \sim
	2^{\norms{\jbfm'}_{1} - \mx(\jbfm,\jbfm') - \my(\jbfm,\jbfm')}\)
	to obtain 
	\begin{align}
		\notag
		\sum_{\lambdabfm' \in \Ical_{\lambdabfm,\jbfm'}}
		2^{\frac{1}{2}(|\jbfm|_1 - |\jbfm'|_1)}
		\left|\ell_{\lambdabfm,\lambdabfm'}\right|
		&\lesssim
		2^{-\tilde{d} (\Mx(\jbfm,\jbfm') +\My(\jbfm,\jbfm'))} 
		2^{-q(\Mx(\jbfm,\jbfm') + \my(\jbfm,\jbfm'))}
		\Fcal^{-(2q + 2\tilde{d})}\\[-0.3cm]
		\notag
		&= 2^{-(q + \tilde{d}) r}
		2^{-(\tilde{d} + q)|\jx - \jx'|}
		2^{-\tilde{d}|\jy - \jy'|}
		2^{(\tilde{d} + q)|\jbfm -\jbfm'|_1} \\
		\label{eq:to_reduce}
		&= 2^{-(\tilde{d} + q)r} 2^{q|\jy -\jy'|}.
	\end{align}
	Therefore, we may bound the weighted row sum by
	\begin{align*}
		\sum_{|\jbfm -\jbfm'|_{\infty} \leq \alpha r
		\frac{\bar{s}}{\nu}}
		\sum_{\lambdabfm' \in \Ical_{\lambdabfm,\jbfm'}}
		2^{\frac{1}{2}(|\jbfm|_1 - |\jbfm'|_1)}
		\left|\ell_{\lambdabfm,\lambdabfm'}\right|
		&\lesssim 
		2^{-(\tilde{d} + q)r}
		\sum_{|\jbfm -\jbfm'|_{\infty} \leq \alpha r
		\frac{\bar{s}}{\nu}}
		2^{q|\jy - \jy'|}\\
		&\lesssim r^2 2^{-(\tilde{d} + q - \max\{0,\, q\alpha
		\nicefrac{\bar{s}}{\nu}\})r}.
	\end{align*}
\end{proof}

Interchanging the coordinate directions immediately implies the following
corollary.

\begin{corollary}
	Assume (with a slight abuse of notation) that
	\((|\jbfm|_{\infty}, |\jbfm'|_{\infty}) = \linebreak (\mx(\jbfm,\jbfm'),
	\My(\jbfm,\jbfm'))\).
	Then, the statement of Lemma \ref{lm:near_field_previous_second} remains
	true.
\end{corollary}

If \((|\jbfm|_{\infty}, |\jbfm'|_{\infty}) = (\Mx(\jbfm,\jbfm'),
\My(\jbfm,\jbfm'))\), on the other hand,
\eqref{eq:to_reduce} reduces to \(2^{-(\tilde{d} + q)r}\).
Therefore, the following corollary holds as well.

\begin{corollary}
	Assume (with a slight abuse of notation) that
	\((|\jbfm|_{\infty}, |\jbfm'|_{\infty}) = \linebreak (\Mx(\jbfm,\jbfm'),
	\My(\jbfm,\jbfm'))\).
	Then, if all entries specified in Lemma \ref{lm:near_field_previous_second}
	are dropped, each appropriately weighted row or column sum of the error
	matrix is of order \(\Ocal(r^2 2^{-(\tilde{d} + q)r})\).
\end{corollary}

Finally, we also need to estimate the error for the second compression when both
\(|\cdot|_{\infty}\)-norms correspond to the \(\mathrm{y}\)-coordinate.

\begin{lemma}
	\label{lm:near_field_previous_third}
	Assume that \((|\jbfm|_{\infty},\, |\jbfm'|_{\infty}) = (\jy,\, \jy')\).
	If all entries are dropped for which
	\begin{align*}
		2^{-\mx(\lambdabfm,\lambdabfm')} &\geq 
		\sigmax(\lambdabfm,\lambdabfm') \geq \Fcal \isdef 
		2^{-\mx(\jbfm,\jbfm')} 2^{\frac{r}{2} - |\jx -\jx'|},\\
		2^{-\my(\jbfm,\jbfm')} &\geq \deltay(\lambdabfm,\lambdabfm'),
	\end{align*}
	then each appropriately weighted row or column sum of
	the error matrix is of order \(\Ocal(2^{-sr})\),
	provided \(s < \tilde{d} + q - \max\left\{0,\,
	q \alpha \nicefrac{\bar{s}}{\nu}\right\}\).
\end{lemma}
\begin{proof}
	Once more, we use \(\norms{\Ical_{\lambdabfm,\jbfm'}} \sim
	2^{\norms{\jbfm}_{1} - \mx(\jbfm,\jbfm') - \my(\jbfm,\jbfm')}\)
	and Lemma \ref{lm:near_field_previous} to obtain that
	\begin{align}
		\notag
		\sum_{\lambdabfm' \in \Ical_{\lambdabfm,\jbfm'}}
		2^{\frac{1}{2}(|\jbfm|_1 - |\jbfm'|_1)}
		\left|\ell_{\lambdabfm,\lambdabfm'}\right|
		&\lesssim
		2^{-\tilde{d} (\Mx(\jbfm,\jbfm') + \My(\jbfm,\jbfm'))} 
		2^{-q (\jy + \jy')}
		\Fcal^{-(2q + 2\tilde{d})}\\[-0.3cm]
		\notag
		&= 2^{-(q + \tilde{d}) r}
		2^{-\tilde{d}(\Mx(\jbfm,\jbfm') + 
		\My(\jbfm,\jbfm') -2 \mx(\jbfm,\jbfm'))}\\
		\label{eq:to_differ_to_bound}
		&\qquad\qquad
		\cdot 2^{-q(\jy + \jy' - 2\mx(\jbfm,\jbfm'))}
		2^{(2q + 2\tilde{d})|\jx -\jx'|}.
	\end{align}
	Due to the identities	
	\begin{align*}
		\My(\jbfm, \jbfm') -\mx(\jbfm,\jbfm')
		&= |\jx -\jx'| + \My(\jbfm,\jbfm')
		-\Mx(\jbfm,\jbfm'),\\
		\jy +\jy' -2\mx(\jbfm,\jbfm')
		&= 2 |\jx -\jx'| - |\jy -\jy'| + 2\big(\My(\jbfm,\jbfm') -
		\Mx(\jbfm,\jbfm') \big),
	\end{align*}
	we may use that fact that 
	\((\tilde{d} + 2q)(\My(\jbfm,\jbfm') -\Mx(\jbfm,\jbfm')) \geq 0\)
	to bound \eqref{eq:to_differ_to_bound} by
	\begin{align*}
		\sum_{\lambdabfm' \in \Ical_{\lambdabfm,\jbfm'}}
		2^{\frac{1}{2}(|\jbfm|_1 - |\jbfm'|_1)}
		\left|\ell_{\lambdabfm,\lambdabfm'}\right|
		&\lesssim 2^{-(q + \tilde{d})r} 
		2^{q|\jy-\jy'|}.
	\end{align*}
	Thus, again, if \(q \geq 0\), there holds
	\begin{align*}
		\sum_{|\jbfm - \jbfm'|_{\infty} \leq \alpha r
		\frac{\bar{s}}{\nu}} 
		\sum_{\lambdabfm' \in \Ical_{\lambdabfm,\jbfm'}}
		2^{\frac{1}{2}(|\jbfm|_1 - |\jbfm'|_1)}
		\left|\ell_{\lambdabfm,\lambdabfm'}\right|
		&\lesssim 
		\sum_{|\jbfm - \jbfm'|_{\infty} \leq \alpha r
		\frac{\bar{s}}{\nu}} 
		2^{-(q + \tilde{d})r} 2^{q|\jy-\jy'|}\\
		&\lesssim
		r^2 2^{-(\tilde{d} + q - q\alpha
		\nicefrac{\bar{s}}{\nu}) r},
	\end{align*}
	whereas if \(q < 0\), 
	\begin{align*}
		\sum_{|\jbfm - \jbfm'|_{\infty} \leq \alpha r
		\frac{\bar{s}}{\nu}} 
		\sum_{\lambdabfm' \in \Ical_{\lambdabfm,\jbfm'}}
		2^{\frac{1}{2}(|\jbfm|_1 - |\jbfm'|_1)}
		\left|\ell_{\lambdabfm,\lambdabfm'}\right|
		&\lesssim 
		r 2^{-(\tilde{d} + q) r}.
	\end{align*}
\end{proof}

By combining Lemmata \ref{lm:near_field_previous_first},
\ref{lm:near_field_previous_second}, and
\ref{lm:near_field_previous_third} with Schur's lemma, 
we can pose the following theorem.

\begin{theorem}
	\label{thm:near_field_previous}
	Let us define the cutoff parameter
	\begin{align*}
		\Fparam{\jbfm}{x} \isdef 
		\begin{cases}
			2^{-\mx(\jbfm,\jbfm')} 2^{\frac{r}{2} - |\jx -\jx'|}, 
			&\text{if } (|\jbfm|_{\infty}, |\jbfm'|_{\infty}) = (\jx,\jx'),\\
			2^{-\mx(\jbfm,\jbfm')} 2^{\frac{r}{2} - |\jx -\jx'|}, 
			&\text{if } (|\jbfm|_{\infty}, |\jbfm'|_{\infty}) = (\jy,\jy'),\\
			2^{-\frac{1}{2}(\mx(\jbfm,\jbfm') + \my(\jbfm,\jbfm'))}
			2^{\frac{1}{2}(r - |\jbfm -\jbfm'|_1)},
			&\text{else},
		\end{cases}
	\end{align*}
	and the matrix \(\Lbfm^{(2, \mathrm{x})}_r\) by the entries
	\begin{equation}
		[\ell_r^{(2,{\rm x})}]_{\lambdabfm,\lambdabfm'}\isdef
		\begin{cases}
			0, &\text{if }[\ell_r^{(\textnormal{mix})}]_{\lambdabfm,\lambdabfm'}=0, \\
			0, &\text{if }
			\begin{cases}
				2^{-\mx(\lambdabfm,\lambdabfm')} \geq
				\sigmax(\lambdabfm,\lambdabfm')\geq \Fparam{\lambdabfm}{x},\\
				2^{-\my(\lambdabfm,\lambdabfm')} \geq \deltay(\lambdabfm,\lambdabfm'),
			\end{cases}\\
			\langle\Lcal\psi_{\lambdabfm'}, \psi_{\lambdabfm}\rangle_\square,
			&\text{else}.
		\end{cases}
		\label{eq:compression_scheme_second_x}
	\end{equation}
	Then, for any \(s < \frac{\tilde{d}}{2} + q - \max\left\{0,\, q\alpha
	\nicefrac{\bar{s}}{\nu} \right\}\), 
	the cutoff error is asymptotically bounded by 
	\(\|\Lbfm_r^{(\operatorname{mix})} - \Lbfm_r^{(2, \mathrm{x})}\| \lesssim 2^{-sr}\).
\end{theorem}

In complete analogy, we define \(\Lbfm^{(2, \mathrm{y})}_r\) to state the
following theorem.

\begin{theorem}
	\label{thm:compression_scheme_second_combined}
	Let us define the matrix \(\Lbfm_r^{(2)}\) by
	\begin{equation*}
		[\ell_r^{(2)}]_{\lambdabfm,\lambdabfm'}\isdef
		\begin{cases}
			0, &\text{if }[\ell_r^{(2, {\rm z})}]_{\lambdabfm,\lambdabfm'}=0
			\text{ for some } \mathrm{z} \in \{\mathrm{x}, \mathrm{y}\}, \\
			\langle\Lcal\psi_{\lambdabfm'}, \psi_{\lambdabfm}\rangle_\square,
			&\text{else}.
		\end{cases}
	\end{equation*}
	Then, for any \(s < \frac{\tilde{d}}{2} + q - \max\left\{0,\, q\alpha
	\nicefrac{\bar{s}}{\nu} \right\}\), 
	there holds \(\|\Lbfm_r^{(\textnormal{mix})} -\Lbfm_r^{(2)}\|_2
	\lesssim 2^{-sr}\).
\end{theorem}

\subsection{Complexity}

Let us now estimate the number of remaining, 
nontrivial matrix entries. 
Also here, we start with the first compression.

\begin{theorem}
	For any \(\lambdabfm \in\nablabfm\),
	consider all entries \([\ell_r]_{\lambdabfm,\lambdabfm'}\)
	for indices satisfying \(\delta(\lambdabfm,\lambdabfm')
	\geq 2^{-\minm(\lambdabfm,\lambdabfm')}\).
	Then, the number of such
	entries in the \(\lambdabfm\)-th row or column
	is asymptotically bounded by \(r2^{r}\).
	\label{thm:complexity_first}
\end{theorem}
\begin{proof}
	We fix \(r\in\Nbbb\) and 
	we note that the number of
	such entries can be estimated by
	\begin{equation*}
		N_r \lesssim \sum_{|\jbfm-\jbfm'|_{\infty} \leq \alpha r
		\frac{\bar{s}}{\nu}}
		2^{|\jbfm'|_1}\Bcal_{\jbfm,\jbfm'}^2,
	\end{equation*}
	where \(\jbfm\) is fixed and denotes the level of \(\lambdabfm\).
	Moreover, we may assume 
	without loss of generality that
	\begin{equation}
		\Bcal_{\jbfm,\jbfm'} = 2^{-\frac{1}{2}(\mx(\jbfm,\jbfm')+\my(\jbfm,\jbfm'))}
		2^{\xi(\frac{r}{2} - |\jbfm-\jbfm'|_{\infty})}
		\label{eq:cutoff_first_not_minimum}
	\end{equation}
	for some \(\nicefrac{1}{2} < \xi < 1\),
	because if \(\Bcal_{\jbfm,\jbfm'} = 2^{-\minm(\jbfm,\jbfm')}\),
	then all such entries are compressed and therefore trivial.
	Hence, by recalling \eqref{eq:cutoff_parameter_first_general}, we may also assume
	that \(|\jbfm-\jbfm'|_{\infty} \leq \nicefrac{r}{2}\).
	Therefore, we have
	\begin{align*}
		N_r &\lesssim \sum_{|\jbfm - \jbfm'|_{_\infty} \leq \frac{r}{2}}
		2^{|\jbfm'|_1} \Bcal_{\jbfm,\jbfm'}^{2}
		\leq 2^{\xi r} \sum_{|\jbfm - \jbfm'|_{_\infty} \leq \frac{r}{2}}
		2^{|\jbfm-\jbfm'|_1-2\xi|\jbfm-\jbfm'|_{\infty}}\\
		&\leq 2^{\xi r} \sum_{|\jbfm - \jbfm'|_{_\infty} \leq \frac{r}{2}}
		2^{2(1-\xi)|\jbfm-\jbfm'|_{\infty}}\\
		&\lesssim r2^r.
	\end{align*}
\end{proof}

In the next step, we will estimate the number of entries in the mixed
field, that is, those entries which are in the far-field in one direction
and in the near-field in the other direction.

\begin{theorem}
	\label{thm:complexity_mixed_previous}
	Consider all entries such that
	\begin{align*}
		\deltay(\lambdabfm,\lambdabfm') \geq 2^{-\my(\lambdabfm,\lambdabfm')},
		\qquad
		\deltax(\lambdabfm,\lambdabfm') \leq
		2^{-\mx(\lambdabfm,\lambdabfm')}.
	\end{align*}
	Then, only \(\Ocal(r 2^r)\) nontrivial entries remain.
\end{theorem}
\begin{proof}
	We know that only entries, for which both,
	\(\sigmax(\lambdabfm,\lambdabfm') \leq \Eparam{\lambdabfm}{x}\)
	and
	\(\deltay(\lambdabfm,\lambdabfm') \leq \Dparam{\lambdabfm}{y}\), 
	remain in the compressed matrix.
	Moreover, we only need to consider the entries which satisfy
	\(|\jy -\jy'| \leq \nicefrac{r}{2}\), 
	since otherwise, in view of \eqref{eq:cutoff_mixed_field_y_far},
	there holds \(\Dparam{\jbfm}{y} = 2^{-\my(\jbfm,\jbfm')}\) and these
	entries are trivial as \(\deltay(\lambdabfm,\lambdabfm') \geq
	2^{-\my(\lambdabfm,\lambdabfm')}\).
	Hence, \eqref{eq:cutoff_mixed_field_y_far} also tells us that
	\(\Dparam{\jbfm}{y} = 2^{-\my(\jbfm,\jbfm')}
	2^{\theta(\nicefrac{r}{2} - \norms{\jy - \jy'})}\),
	where \(\nicefrac{1}{2} < \theta < 1\).

	In what follows, we will also need to make a distinction in 
	\(|\jx -\jx'|\). 
	First, if \(|\jx -\jx'| \leq \nicefrac{r}{2}\),
	we do not need to perform a second compression.
	The number of remaining entries can therefore be estimated by
	\begin{align*}
		&\sum_{|\jx -\jx'| \leq \frac{r}{2}}
		\sum_{|\jy -\jy'| \leq \frac{r}{2}}
		2^{|\jbfm'|_1} 2^{-\mx(\jbfm,\jbfm')} \Dparam{\jbfm}{y}\\
		&\qquad\qquad
		= 2^{\theta \frac{r}{2}}
		\sum_{|\jx -\jx'| \leq \frac{r}{2}} 2^{|\jx -\jx'|}
		\sum_{|\jy -\jy'| \leq \frac{r}{2}}
		2^{(1-\theta)|\jy -\jy'|}\\
		&\qquad\qquad
		\lesssim 2^{r},
	\end{align*}
	since \(\theta < 1\).

	On the other hand, if \(|\jx -\jx'| \geq \nicefrac{r}{2}\), 
	we perform a second compression as well
	with the parameter \(\Eparam{\jbfm}{x} =
	2^{-\mx(\jbfm,\jbfm')} 2^{\nicefrac{r}{2} - \norms{\jx - \jx'}}\)
	from \eqref{eq:Eparam_previous}.
	Hence, the number of nontrivial entries is bounded by
	\begin{align*}
		&\sum_{\frac{r}{2} \leq |\jx -\jx'| \leq 
		\alpha r \frac{\bar{s}}{\nu}} \
		\sum_{|\jy -\jy'| \leq \frac{r}{2}}
		2^{|\jbfm'|_1} \Eparam{\jbfm}{x} \Dparam{\jbfm}{y}\\
		&\qquad\qquad
		= 2^{(1 + \theta)\frac{r}{2}}
		\sum_{\frac{r}{2} \leq |\jx -\jx'| \leq 
		\alpha r \frac{\bar{s}}{\nu}} \
		\sum_{|\jy -\jy'| \leq \frac{r}{2}}
		2^{(1-\theta)|\jy -\jy'|}\\
		&\qquad\qquad
		\lesssim r 2^{r}.
	\end{align*}
\end{proof}

Finally, we also want to estimate the number of nontrivial entries 
in the near-field, meaning that both \(\deltax(\lambdabfm,\lambdabfm')
\leq 2^{-\mx(\lambdabfm,\lambdabfm')}\) and \(\deltay(\lambdabfm,\lambdabfm')
\leq 2^{-\my(\lambdabfm,\lambdabfm')}\) hold.
In view of Theorem \ref{thm:compression_scheme_second_combined},
we again need to distinguish between three cases, 
where we shall start with the easiest one.

\begin{lemma}
	\label{lm:complexity_near_second}
	Assume (with a slight abuse of notation) that 
	\((|\jbfm|_{\infty},\, |\jbfm'|_{\infty}) 
	= \linebreak (\Mx(\jbfm,\jbfm'),\, \my(\jbfm,\jbfm'))\),
	\((|\jbfm|_{\infty},\, |\jbfm'|_{\infty}) 
	= (\mx(\jbfm,\jbfm'),\, \My(\jbfm,\jbfm'))\),
	or \((|\jbfm|_{\infty},\, |\jbfm'|_{\infty}) 
	= \linebreak (\Mx(\jbfm,\jbfm'),\, \My(\jbfm,\jbfm'))\).
	Then, the number of nontrivial near-field entries in the \(\lambdabfm\)-th row or
	column is of order \(r^2 2^r\).
\end{lemma}
\begin{proof}
	Fix \(\lambdabfm \in \nablabfm_{\jbfm}\).
	If \(|\jbfm -\jbfm'|_1 \leq r\), 
	we do not need to perform a second compression,
	since in this case, the number of nontrivial near-field entries entries 
	in the \(\lambdabfm\)-th row or column is trivially bounded by 
	\(\Ocal(r 2^r)\).

	On the other hand, if \(|\jbfm -\jbfm'|_1 \geq r\), we perform a second
	compression in both directions. 
	As the parameter in Theorem \ref{thm:near_field_previous} is symmetric
	in the current case,
	we can control the number of nontrivial near-field entries
	in the \(\lambdabfm\)-th row or column by
	\begin{align*}
		\sum_{|\jbfm -\jbfm'|_1 \geq r} 2^{|\jbfm'|_1} 
		\Fparam{\jbfm}{x}\Fparam{\jbfm}{y}
		&= \sum_{|\jbfm -\jbfm'|_1 \geq r} 2^{|\jbfm'|_1} 
		2^{-\my(\jbfm,\jbfm') -\my(\jbfm,\jbfm')} 2^{r -
		\norms{\jbfm-\jbfm'}_1} \\
		&\lesssim 2^{r} \sum_{|\jbfm -\jbfm'|_1 \geq r}
		1.
	\end{align*}
	Since we only need to consider entries with \(\norms{\jbfm -
	\jbfm'}_{\infty} \leq \alpha r \nicefrac{\bar{s}}{\nu}\),
	the above sum can be asymptotically bounded by \(r^2 2^r\).
\end{proof}

\begin{lemma}
	\label{lm:complexity_near_first}
	Assume that \((|\jbfm|_{\infty},\, |\jbfm'|_{\infty}) = (\jx,\, \jx')\).
	Then, the number of nontrivial near-field entries in the \(\lambdabfm\)-th row or
	column is of order \(r^2 2^r\).
\end{lemma}
\begin{proof}
	We first remark that if \(|\jbfm -\jbfm'|_{\infty} \leq
	\nicefrac{r}{2}\), the statement of the theorem is trivially true.

	If \(|\jx -\jx'| \geq \nicefrac{r}{2}\) and \(|\jy -\jy'| \leq
	\nicefrac{r}{2}\), we perform a second compression in
	\(\mathrm{x}\) only
	with the parameter \(\Fparam{\jbfm}{x} =
	2^{-\mx(\jbfm,\jbfm')} 2^{\nicefrac{r}{2} - \norms{\jx - \jx'}}\) from
	Theorem \ref{thm:near_field_previous}. 
	Therefore, the number of nontrivial entries on level \(\jbfm'\) is bounded by
	\begin{equation*}
		N_{\lambdabfm,\jbfm'} \lesssim 2^{|\jbfm'|_1} 2^{-\my(\jbfm,\jbfm')}
		\Fparam{\jbfm}{x}
		\lesssim 2^{\frac{r}{2}} 2^{|\jy -\jy'|},
	\end{equation*}
	so summing over all such \(\jbfm'\) yields the estimate
	\begin{equation*}
		\sum_{\frac{r}{2} \leq |\jx -\jx'| \leq 
		\alpha r \frac{\bar{s}}{\nu}} \
		\sum_{|\jy -\jy'| \leq \frac{r}{2}} 
		N_{\lambdabfm,\jbfm'} \lesssim r 2^r.
	\end{equation*}

	If \(|\jx -\jx'| \leq \nicefrac{r}{2}\) and 
	\(|\jy -\jy'| \geq \nicefrac{r}{2}\),
	we perform a second compression in \(\mathrm{y}\) only. 
	Here, we need to exchange the indices in the cutoff parameter
	from Theorem \ref{thm:near_field_previous}, i.e., 
	we use the parameter 
	\(\Fparam{\jbfm}{y} = 2^{-\my(\jbfm,\jbfm')} 
	2^{\nicefrac{r}{2} - \norms{\jy - \jy'}}\). 
	This results in
	\begin{align*}
		N_{\lambdabfm,\jbfm'} \lesssim 2^{|\jbfm'|_1} 2^{-\mx(\jbfm,\jbfm')}
		\Fparam{\jbfm}{y}
		&\lesssim 2^{\frac{r}{2}} 2^{|\jx -\jx'|},
	\end{align*}
	which sums up in accordance with
	\begin{align*}
		\sum_{|\jx  - \jx'| \leq \frac{r}{2}} \
		\sum_{\frac{r}{2} \leq |\jy - \jy'| \leq 
		\alpha r \frac{\bar{s}}{\nu}} 
		N_{\lambdabfm,\jbfm'}
		\lesssim 2^{\frac{r}{2}} \sum_{|\jx - \jx'| \leq \frac{r}{2}} \
		\sum_{\frac{r}{2} \leq |\jy - \jy'| \leq 
		\alpha r \frac{\bar{s}}{\nu}} 2^{|\jx - \jx'|}
		\lesssim r 2^r.
	\end{align*}

	Finally, if both \(|\jx -\jx'| \geq \nicefrac{r}{2}\) and 
	\(|\jy -\jy'| \geq \nicefrac{r}{2}\), 
	we perform a second compression in both coordinate directions, 
	and therefore,
	\begin{align*}
		N_{\lambdabfm,\jbfm'} \lesssim 2^{|\jbfm'|_1} 
		\Fparam{\jbfm}{x}\Fparam{\jbfm}{y}
		&\lesssim 2^{r},
	\end{align*}
	which sums up to
	\begin{align*}
		\sum_{\frac{r}{2} \leq |\jx - \jx'| \leq 
		\alpha r \frac{\bar{s}}{\nu}} \
		\sum_{\frac{r}{2} \leq |\jy - \jy'| \leq 
		\alpha r \frac{\bar{s}}{\nu}} N_{\lambdabfm,\jbfm'}
		\lesssim 2^{r} 
		\sum_{\frac{r}{2} \leq |\jx - \jx'| \leq 
		\alpha r \frac{\bar{s}}{\nu}} \
		\sum_{\frac{r}{2} \leq |\jy - \jy'| \leq 
		\alpha r \frac{\bar{s}}{\nu}} 1
		\lesssim r^2 2^r
	\end{align*}
	as well.
\end{proof}

The same calculations, but interchanged coordinate directions, show us also the
following:
\begin{corollary}
	\label{cor:complexity_near_third}
	Assume that \((|\jbfm|_{\infty}, |\jbfm'|_{\infty}) = (\jy,\jy')\).
	Then, the number of nontrivial near-field entries in the \(\lambdabfm\)-th row or
	column is of order \(r^2 2^r\).
\end{corollary}

By combining Lemmata \ref{lm:complexity_near_second},
\ref{lm:complexity_near_first}, and Corollary \ref{cor:complexity_near_third},
we finally arrive at the following theorem:

\begin{theorem}
	\label{thm:complexity_near_previous}
	Consider all entries such that
	\begin{align*}
		\deltax(\lambdabfm,\lambdabfm') \leq 2^{-\mx(\lambdabfm,\lambdabfm')},
		\qquad
		\deltay(\lambdabfm,\lambdabfm') \leq 2^{-\my(\lambdabfm,\lambdabfm')}.
	\end{align*}
	Then, the number of entries is bounded by \(\Ocal(r^2 2^r)\).
\end{theorem}

By combining the complexity estimates with the error estimates 
earlier in this section, we can state our main theorem.

\begin{theorem}
	\label{thm:s_star_compressibility}
	Consider a continuous and elliptic operator 
	\(\Lcal\colon H^{q}(\square) \to H^{-q}(\square)\) with
	an asymptotically smooth kernel.
	Then, the operator matrix corresponding to a wavelet basis of order
	\(d\) with regularity \(\gamma\) and \(\tilde{d}\) vanishing moments
	is \(s^{\star}\)-compressible with
	\begin{align*}
		s^{\star} = \min\left\{\sigma,\ 
			\alpha \bar{s},\
			\frac{\tilde{d}}{2} + q - \max\left\{0, \ q\alpha
			\frac{\bar{s}}{\nu(\gamma, q)} \right\}
		\right\}.
	\end{align*}
\end{theorem}

\begin{remark}
	\label{rm:sparse_grid_rate}
	Let us remark again that, if \(\sigma\) is sufficiently large,
	we can always achieve \(s^{\star} > \bar{s}\) 
	by choosing a wavelet basis with sufficiently many vanishing moments.
	As a consequence, using the procedures of \cite{CDD01,CDD02}, 
	we may approximate \(u\) at the optimal rate
	\(N^{-\bar{s}}\),
	provided the relevant entries of the operator matrix \(\Lbfm\) are known 
	or the can be computed in \(\Ocal(N)\) computing time.
	We also note that this is the same rate which can be achieved by sparse
	grid approximation, cf.\ \cite{BG04,GH11,GK00,GOS99},
	provided the target function admits sufficient regularity.
\end{remark}

\begin{remark}
	\label{rm:number_of_van_mom_necessary}
	In order to reach the best possible rate,
	for the following, classical combinations of \(d\) and \(q\),
	the minimal number of vanishing moments is tabulated below:
	\begin{equation*}
		\begin{array}{c || c | c }
			& d = 1 & d = 2 \\
			\hline
			\hline
			2q = - 1 & 4 & 6 \\
			\hline
			2q = 0 & 3 & 5 \\
			\hline
			2q = 1 & \text{--} & 4 \\
		\end{array}
	\end{equation*}
	For these values, we have assumed that the wavelet bases have the maximally
	possible global smoothness, i.e., that \(\gamma = d -\nicefrac{1}{2}\). 
	Moreover, we note that using piecewise constant wavelets is not meaningful
	in the case \(2q = 1\). 
	Finally, for piecewise linear wavelets and \(2q = 1\), 
	we have assumed \(1 < \alpha < \nicefrac{4}{3}\).
\end{remark}

\begin{remark}
	If we just intend to receive the maximum convergence rate of an isotropic 
	discretisation \(N^{-\frac{d-q}{2}}\), which nevertheless resolves edge 
	singularities, we can slightly adjust the compression scheme. 
	If we compress in Theorem \ref{thm:diagonal_cutoff} 
	matrix entries satisfying \(|\jbfm -\jbfm'|_{\infty} > r\) 
	in \eqref{eq:compression_diagonal_cutoff}, 
	then the \(s^\star\)-compressibility is limited by 
	\(s^\star \leq \min\{\sigma,\, \gamma-q\}\). 
	This bound is also present in the isotropic case, see 
	\cite{Ste04}.

	In turn, this means that we only need to sum over
	all entries with \(|\jbfm-\jbfm'|_{\infty} \leq r\) in Lemma
	\ref{lm:near_field_previous_first}, by which we get the 
	additional bound \(s^\star \leq \nicefrac{\tilde{d}}{2} + q - \max\{0,q\}\). 
	Combined, this results in \(s^{\star} = \min\{ \sigma,\, \gamma-q,\,
	\nicefrac{\tilde{d}}{2} + q - \max\{0, q\} \}\).

	Compared to the isotropic case, the limiting condition 
	on \(s^{\star}\) is \(\nicefrac{\tilde{d}}{2}\) if \(q \geq 0\). This is larger 
	than the desired rate \(s^\star>\frac{d-q}{2}\) whenever \(\tilde{d} > d-q\).
	If \(q < 0\), then we need to have \(\frac{\tilde{d} + 2q}{2} > \frac{d-q}{2}\), 
	or equivalently, \(\tilde{d} > d-3q\). Therefore, wavelets are required 
	with \(\tilde{d} > \max\{d-q,\, d - 3q\}\) vanishing moments to ensure 
	\(s^\star>\frac{d-q}{2}\). This condition somewhat differs from 
	\cite{DHS06,HvR24} for the nonadaptive case, but is satisfied 
	for the typical wavelets used to discretise the classical 
	boundary integral operators.
\end{remark}

\section{On the Generalisation to Lipschitz Manifolds}
\label{sec:manifold}

In Section \ref{sec:compression_general}, we have constructed a compression 
scheme for operators with an asymptotically smooth kernel on the unit square. 
In many problems from engineering or physics, 
one however needs to consider an integral equation 
\(\mathcal{L}u = g\in H^{-q}(\Gamma)\) on the boundary of a domain  
\(\Omega \subset \Rbbb^3\) with Lipschitz boundary 
\(\Gamma \isdef\partial \Omega\), 
resulting in the use of the boundary element method \cite{SS11,Ste08}. 
As it was done for the nonadaptive situation in \cite{HvR24}, 
we will generalise the compression scheme from the unit 
square to a Lipschitz manifold.

\subsection{Surface Representation}

We assume that the boundary \(\Gamma\) can be decomposed into \(P\) four-sided, 
smooth patches \(\Gamma_i\), \(i = 1, \dots, P\). 
It is required that this decomposition is admissible, i.e., 
for \(i \neq i'\), the intersection \(\Gamma_i \cap \Gamma_{i'}\) 
is either empty or a common vertex or edge of both \(\Gamma_i\) and
\(\Gamma_{i'}\).
We further assume that, for each \(1 \leq i \leq P\), there exits a smooth 
diffeomorphism \(\tauvec_i \colon \square \to \Gamma_i\) such that 
there are constants \(0 < c_i \leq C_i < \infty\) with
\begin{align}
	\label{eq:weight_per_patch}
	c_i \leq \omega_i(\hat{\xvec}) 
	\isdef \sqrt{\Det \big(\Dbfm \tauvec_i(\hat{\xvec})^{\intercal}
	\Dbfm \tauvec_i(\hat{\xvec})\big)}
	\leq C_i, \qquad \hat{\xvec} \in \square.
\end{align}
Moreover, on a common edge of \(\Gamma_i\) and \(\Gamma_{i'}\), 
we assume that the two parametrisations \(\tauvec_i\) and 
\(\tauvec_{i'}\) coincide up to orientation.

Along the lines of \cite{SS11,Ste08}, we can introduce the Sobolev spaces
\(H^{s}(\Gamma)\) for \(|s| \leq s_{\Gamma}\), 
where \(s_{\Gamma}\) depends on the regularity of \(\Gamma\).
In the case of a \(C^{k,\alpha}\)-manifold, one has 
\(s_{\Gamma} = k + \alpha\), cf.\ \cite{Wlo87}, thus
we have \(s_{\Gamma} = 1\) for a Lipschitz manifold. 

We note that there holds \(|q| \leq \nicefrac{1}{2}\) for the 
classical boundary integral operators arising from 
second order partial differential equations. Hence, for 
each of them, the energy space \(H^{q}(\Gamma)\) can be 
characterised on a Lipschitz manifold. In particular, 
note that with \eqref{eq:weight_per_patch}, 
the inner product
\begin{align*}
	\langle u, v \rangle_{\Gamma} = \int_{\Gamma} u(\xvec) v(\xvec) \dint
	S_{\xvec}
	= \sum_{i=1}^{P} \int_{\square} u\big(\tauvec_i(\hat{\xvec})\big) 
	v\big(\tauvec_i(\hat{\xvec})\big) \omega_i(\hat{\xvec}) \dint \hat{\xvec}
\end{align*}
can be extended to the duality pairing in \(H^{-s}(\Gamma) \times
H^{s}(\Gamma)\) and that on each patch, up to a smooth weight 
function, it is the transported duality pairing on the unit square.

\subsection{Discontinuous Wavelet Bases}

If \(q < \nicefrac{1}{2}\), then the wavelet basis does not need to be 
continuous to ensure that \(\gamma > q\). 
In this case, the construction of a basis is rather straightforward.
Since every patch \(\Gamma_i\), \(1 \leq i \leq P\), 
is the image of the unit square under a smooth diffeomorphism, 
we may simply transport the wavelets onto the geometry.
In particular, a multiindex \(\lambdabfm = (i, \jbfm, \kbfm)\) now encodes 
information on the patch, the level, and the location of its support on the
respective patch.
With this, we define \(\psi_{\lambdabfm} \isdef \hat{\psi}_{\hat{\lambdabfm}} \circ
\tauvec_i^{-1}\), where \(\hat{\psi}_{\hat{\lambdabfm}}\) denotes 
the corresponding wavelet on the unit square for 
\(\hat{\lambdabfm} = (\jbfm,\kbfm)\).
Hence, if \(v \in H^{s}(\Gamma)\) for some \(s > -\gamma\), we have
\begin{align*}
	\big\langle \psi_{\lambdabfm}, v \big\rangle_{\Gamma} = 
	\int_{\Gamma_i} \psi_{\lambdabfm}(\xvec) v(\xvec) \dint S_{\xvec}
	= \int_{\square} \hat{\psi}_{\hat{\lambdabfm}}(\hat{\xvec})
	v\big(\tauvec_i(\hat{\xvec})\big) \omega_i(\hat{\xvec}) \dint \xvec.
\end{align*}
Since \(\tauvec_i\) and, consequently, also \(\omega_i\) are smooth,
the usual cancellation estimates can be concluded from those on the 
unit square in a straightforward manner.

\subsection{Continuous Wavelet Bases}

In contrast, if \(q \geq \nicefrac{1}{2}\), we need to consider a wavelet
basis of globally continuous functions.
In this case, we denote \(\lambdabfm \isdef (\Xi, \jbfm, \kbfm)\), 
where we require that the wavelet basis under consideration consists of
wavelets that satisfy one of the following conditions:
\begin{itemize}
	\item The wavelet is supported on precisely one vertex \(\xivec\) and the
		neighbouring patches. 
		In this case, we denote \(\Xi \isdef \xivec\).
	\item The wavelet is supported on precisely one edge \(\Sigma\) and
		the two neighbouring patches.
		In this case, we denote \(\Xi \isdef \Sigma\).
	\item The wavelet is supported on precisely one patch \(\Gamma_i\). 
		In this case, we denote \(\Xi \isdef i\).
\end{itemize}

For all of the wavelets involved, we assume that we have a 
patchwise cancellation property, meaning that all elementary 
wavelet estimates used on the unit square carry over to the 
manifold:
\begin{assumption}
	\label{ass:patchwise_cancellation_property}
	For \(\jbfm \geq \jbfm_0\) and \(\lambdabfm \in \nablabfm_{\jbfm}\)
	let us denote \(\hat{\psi}_{\hat{\lambdabfm}}^{(i)} \isdef
	\psi_{\lambdabfm} \circ \tauvec_i\).
	Moreover, for \(v\colon \Gamma \to \Rbbb\) with 
	\(\Supp v \cap \Gamma_i \neq \emptyset\), 
	denote \(\hat{v} \isdef v \circ \tauvec_i\).
	Then, we assume that for
	\begin{align*}
		\langle \psi_{\lambdabfm}, v \rangle_{\Gamma_i} = \int_{\square}
		\hat{\psi}_{\hat{\lambdabfm}}^{(i)}(\hat{\xvec}) \hat{v}(\hat{\xvec})
		\omega_i(\hat{\xvec}) \dint \hat{\xvec}
	\end{align*}
	the usual cancellation estimates hold, i.e., we assume that
	\(\hat{\psi}_{\hat{\lambdabfm}}^{(i)}\) can be treated like a wavelet of level
	\(\jbfm\) with \(\tilde{d}\) vanishing moments on the unit square.
\end{assumption}

Note that, in the case of isotropic wavelets, such bases have been 
constructed in \cite{HS06}. As a consequence of Assumption
\ref{ass:patchwise_cancellation_property}, in view of
\begin{align*}
	\big\langle \psi_{\lambdabfm}, v \big\rangle_{\Gamma}
	= \sum_{i=1}^{P} \big\langle \psi_{\lambdabfm}, 
	v \big\rangle_{\Gamma_i}
	&= \sum_{i=1}^{P} \int_{\square}
	\hat{\psi}_{\hat{\lambdabfm}}^{(i)}(\hat{\xvec}) \hat{v}(\hat{\xvec})
	\omega_{i}(\hat{\xvec}) \dint \hat{\xvec},
\end{align*}
we may treat such duality pairings like a sum of duality pairings on the unit
square, where for each summand,
we are allowed to exploit the usual decay properties of the wavelet involved.

\subsection{Matrix Entry Estimates}

As discussed in \cite{HvR24}, we can bound the entries on the
geometry using the usual decay estimates as well.
We keep in mind that the following procedure also covers the 
case of a globally continuous wavelet basis 
since we have a cancellation property on each patch
because of Assumption \ref{ass:patchwise_cancellation_property}.
Therefore, we may treat every matrix entry as the sum of matrix entries
corresponding to a discontinuous wavelet basis.
However, in order to keep the estimates from Section
\ref{sec:compression_general} valid, 
we also need to assume the following:
\begin{assumption}
	\label{ass:number_of_wavelets}
	Consider two sets \(E, F \subset [0,1]\) and a multiindex \(\jbfm \geq
	\jbfm_0\).
	For \(1 \leq i \leq P\), we assume that
	\begin{itemize}
		\item for any \(x \in E\), there are at most \(\Ocal(1 +
			2^{\jy} |F|)\)
			wavelets supported on \(\tauvec_i(\{x\} \times F)\),
		\item for any \(y \in F\), there are at most \(\Ocal(1 +
			2^{\jx} |E|)\)
			wavelets supported on \(\tauvec_i(E \times \{y\})\),
		\item in total, there are at most \(\Ocal((1 + 2^{\jx} |E|) 
			(1 + 2^{\jy} |F|))\) 
			wavelets supported on \(\tauvec_i(E \times F)\).
	\end{itemize}
\end{assumption}
We will recall only the basic concepts here, 
the most important being the interpretation of the coordinate 
directions over vertices and edges.
For the details, we refer to \cite{HvR24}.

\subsubsection{Wavelets Supported on the Same Patch}

If we consider two wavelets \(\psi_{\lambdabfm}\) and \(\psi_{\lambdabfm'}\)
which are supported on the same patch, the corresponding matrix entry reads as
\begin{align*}
	\big\langle \Lcal \psi_{\lambdabfm'}, \psi_{\lambdabfm} \big\rangle
	= \int_{\square} \int_{\square} \hat{\kappa}_{i,i'} (\hat{\xvec},
	\hat{\xvec}') 
	\hat{\psi}_{\hat{\lambdabfm}}(\hat{\xvec})
	\hat{\psi}_{\hat{\lambdabfm}'}(\hat{\xvec}') 
	\dint \hat{\xvec} \dint \hat{\xvec}',
\end{align*}
where
\begin{align*}
	\hat{\kappa}_{i,i'} (\hat{\xvec}, \hat{\xvec}') 
	\isdef \kappa\big(\tauvec_i(\hat{\xvec}), \tauvec_i(\hat{\xvec}')\big)
	\omega_i(\hat{\xvec}) \omega_i(\hat{\xvec}').
\end{align*}
Since the modified kernel \(\hat{\kappa}_{i,i}\) admits the same decay 
behaviour as \(\kappa\),
we can use \eqref{eq:weight_per_patch} and the Lipschitz continuity of
\(\tauvec_i\) to conclude that such entries can 
be estimated as on the unit square.
However, depending on the geometry, the constants involved may become larger.

\subsubsection{Wavelets Supported on Patches With a Common Edge}

If we consider two wavelets \(\psi_{\lambdabfm}\) and \(\psi_{\lambdabfm'}\), 
where \(\Supp \psi_{\lambdabfm} \subset \Gamma_i\),
\(\Supp \psi_{\lambdabfm'} \subset \Gamma_{i'}\) and \(\Gamma_i\) and
\(\Gamma_{i'}\) share a common edge \(\Sigma\), 
we need to interpret the coordinate directions on \(\Gamma_i \cup
\Gamma_{i'}\).
Without loss of generality, we may assume that the common edge 
satisfies \(\Sigma = \tauvec_i(\{1\} \times [0,1]) = \tauvec_{i'}(\{0\} \times
[0,1])\).
By translating \(\tauvec_{i'}\) by \((1, 0)^{\intercal}\) and
glueing the two parametrisations together, we obtain a Lipschitz continuous
parametrisation \(\tauvec\colon \square^{\Sigma} \isdef
[0,2] \times [0,1] \to \Gamma_i \cup \Gamma_{i'}\). 
Therefore, it is natural to treat the direction
\(\mathrm{x}\) as the direction across the edge \(\Sigma\) and
\(\mathrm{y}\) as the direction parallel to \(\Sigma\).
By the Lipschitz continuity of \(\tauvec\), 
we know that the distance on \(\square^{\Sigma}\) is equivalent to the geodesic
distance on \(\Gamma_i \cup \Gamma_{i'}\) and thus to the euclidean distance in
\(\Rbbb^3\).

Meanwhile, the corresponding matrix entry can be written as
\begin{align*}
	\big\langle \Lcal \psi_{\lambdabfm'}, \psi_{\lambdabfm} \big\rangle
	&= \int_{\square} \int_{\square} \kappa\big( \tauvec_i(\hat{\xvec}),
	\tauvec_{i'}(\hat{\xvec}') \big)
	\hat{\psi}_{\hat{\lambdabfm}}(\hat{\xvec})
	\hat{\psi}_{\hat{\lambdabfm}'}(\hat{\xvec}')
	\omega_{i}(\hat{\xvec}) \omega_{i'}(\hat{\xvec}')
	\dint\hat{\xvec} \dint\hat{\xvec}' \\
	&= \int_{\square^{\Sigma}} \int_{\square^{\Sigma}}
	\hat{\kappa}_{i,i'}(\hat{\xvec}, \hat{\xvec}')
	\hat{\psi}_{\hat{\lambdabfm}}(\hat{\xvec})
	\hat{\psi}_{\hat{\lambdabfm}'}\big(\hat{\xvec}' - 
		\big[
			\begin{smallmatrix}
				1\\
				0\\
			\end{smallmatrix}
		\big]
	\big) \dint\hat{\xvec} \dint\hat{\xvec}',
\end{align*}
where
\begin{align*}
	\hat{\kappa}_{i,i'}(\hat{\xvec}, \hat{\xvec}')
	\isdef \kappa\big(\tauvec(\hat{\xvec}), \tauvec(\hat{\xvec}')\big)
	\omega_{i}(\hat{\xvec}) \omega_{i'}\big(\hat{\xvec}' -
		\big[
			\begin{smallmatrix}
				1\\
				0\\
			\end{smallmatrix}
		\big]
	\big).
\end{align*}
Moreover, we remark that in the current setting, we have \(\hat{\xvec} \in
\square\) and \(\hat{\xvec}' \in [1, 2] \times [0,1]\).
As it was shown in \cite{HvR24}, such matrix entries can be compressed in the
same manner as on the unit square. 
Indeed, if \(\delta(\lambdabfm, \lambdabfm') > 0\), 
then we may apply the first compression provided \(\delta(\lambdabfm,\lambdabfm')\)
is sufficiently large since \(\tauvec\) is piecewise smooth 
and the distances are equivalent.

On the other hand, if we want to apply a second compression in the direction of
\(\mathrm{x}\), we have two situations: 
If the two singular supports touch each other, 
the second compression cannot be applied.
If not, then, 
since the wavelets under consideration are located in different patches, 
we know that for the decomposition of the larger wavelet into \(\fxt\) and
\(\fxb\), we have \(\fxt \equiv 0\).
Therefore, we only need to consider the complement function \(\fxb\),
for which the same estimates as on the unit square can be applied.

Last, for the second compression in \(\mathrm{y}\), 
we remark that both smooth parts of \(\tauvec\) agree on \(\Sigma\).
Therefore, recalling the proof of Lemma \ref{lm:near_field_previous},
we see that for \(\hat{\xvec} \in \square\) and 
\(\hat{\xvec}' \in [1,2] \times [0,1]\),
all derivatives of the modified kernel \(\tilde{\kappa}_{\mathrm{y}}\), 
which depends on \(y\) and \(y'\) only,
exist and that, for \(y \neq y'\), 
they admit the same asymptotic decay behaviour as on the unit square.

\subsubsection{Wavelets Supported on Patches With a Common Vertex}

In the case where two wavelets \(\psi_{\lambdabfm}\) and \(\psi_{\lambdabfm'}\)
are supported on two patches \(\Gamma_i\) and \(\Gamma_{i'}\) which share a
common vertex \(\xivec\),
for which we may without loss of generality assume that \(\xivec =
\tauvec_i((1, 1)^{\intercal}) = \tauvec_{i'}((0, 0)^{\intercal})\),
we translate \(\tauvec_{i'}\) by \((1, 1)^{\intercal}\) and glue the two
parametrisations together.
The continuation onto (possibly multiple) neighbouring patches 
results in a Lipschitz continuous parametrisation \(\tauvec\colon
\square^{\xivec} \isdef [0, 2]^2 \to \Gamma\), such that
\(\tauvec\big|_{\square} = \tauvec_{i}\) and \(\tauvec\big|_{[1,2]^2} =
\tauvec_{i'}(\cdot - (1, 1)^{\intercal})\).

Also here, we remark that the distances on \(\square^{\xivec}\) 
and \(\Rbbb^3\) are equivalent by the Lipschitz continuity of \(\tauvec\).
Hence, we can apply the first compression as on the parameter domain.
Moreover, similar to the second compression across an edge, 
we remark that the two supports either touch each other, 
in which case we cannot compress at all,
or that it suffices to consider the complement function \(\fxb\) or \(\fyb\)
in the decomposition of the larger wavelet, respectively.
Since this holds true for both coordinate directions, 
we may choose the better one and argue as on the unit square as well.

\subsubsection{Wavelets Supported on Well-Separated Patches}

In the case of well-separated patches, we do only consider 
the first compression, which is enough in this case. Indeed, 
after a suitable rescaling of the geometry, we can assume that
\begin{align*}
	\min\big\{\! \Dist(\Gamma_i, \Gamma_{i'}) : 
	\Gamma_i \cap \Gamma_{i'} = \emptyset \big\} 
	\geq 1.
\end{align*}
In particular, in view of \eqref{eq:cutoff_parameter_first_general},
we need to distinguish two cases:
If \(\Bcal_{\jbfm,\jbfm'} = 2^{-\minm(\jbfm,\jbfm')}\), 
we can compress all such entries and if not, then there at most
\(\Ocal(r 2^r)\) such entries in every row or column due to Theorem
\ref{thm:complexity_first} and Assumption \ref{ass:number_of_wavelets}.

\begin{remark}
	In accordance with \cite{Cos88} and the references therein,
	the shift property of the classical boundary integral operators 
	is limited by \(\sigma = \nicefrac{1}{2}\) if \(\Gamma\) is only 
	Lipschitz smooth. 
	In this case, Theorem \ref{thm:s_star_compressibility}
	also limits the maximal \(s^{\star}\) to \(\nicefrac{1}{2}\).
	Nonetheless, since we consider a piecewise smooth Lipschitz surface
	\(\Gamma\),
	the classical boundary integral operators are continuous as operators
	\(H^{q + s}(\Gamma) \to H^{-q + s}(\Gamma)\) for all \(-\nicefrac{1}{2} < s \leq
	\sigma_{0}\), where \(\sigma_{0} > \nicefrac{1}{2}\) 
	depending on the global smoothness of \(\Gamma\), 
	cf.\ \cite[Remark 3.1.18]{SS11}.
	Note that the above restriction is also present when isotropic wavelets are used,
	cf.\ \cite{Ste04}. 
\end{remark}

\subsection*{Acknowledgement}

This research has been supported by the Swiss National Science 
Foundation (SNSF) through the project “Adaptive Boundary Element 
Methods Using Anisotropic Wavelets” under Grant Agreement No.\ 
200021\_192041.

\bibliographystyle{plain}
\bibliography{literature}

\end{document}